\documentclass[11pt]{amsart}
\usepackage{adjustbox}
\usepackage{amssymb}
\usepackage{amsmath}
\usepackage{fancyhdr}
\usepackage[british]{babel}
\usepackage{geometry,mathtools}
\usepackage{enumitem}
\usepackage{algpseudocode}
\usepackage{dsfont}
\usepackage{centernot}
\usepackage{xstring}
\usepackage{colortbl}
\usepackage[section]{algorithm}
\usepackage{graphicx}
\usepackage[font={footnotesize}]{caption}
\usepackage[usenames,dvipsnames,table]{xcolor}
\usepackage{tikz}
\usepackage[h]{esvect}
\usepackage[
		bookmarksopen=true,
		bookmarksopenlevel=1,
		colorlinks=true,
		linkcolor=darkblue,
        linktoc=page,
		citecolor=darkblue,
]{hyperref}
\usepackage{bbm}
\usepackage{mathrsfs}
\usepackage{hyphenat}
\usepackage{soul}
\usetikzlibrary{arrows.meta, decorations.pathmorphing, decorations, fit}

\allowdisplaybreaks

\title[Fractional cycle decompositions]{Fractional cycle decompositions in hypergraphs}
\author[F.~Joos]{Felix Joos}
\address[F.~Joos]{Institut f\"ur Informatik, Universit\"at Heidelberg, 
Germany
}
\email{joos@informatik.uni-heidelberg.de}

\author[M.~K\"uhn]{Marcus K\"uhn}
\address[M.~K\"uhn]{Institut f\"ur Informatik, Universit\"at Heidelberg,
	Germany
}
\email{kuehn@informatik.uni-heidelberg.de}

\thanks{The research leading to these results was partially supported by the Deutsche Forschungsgemeinschaft (DFG, German Research Foundation) -- 428212407.
}

\newtheorem{theorem}[algorithm]{Theorem}

\newtheorem{lemma}[algorithm]{Lemma}

\theoremstyle{definition}

\newtheorem{conj}[algorithm]{Conjecture}

\newtheorem{example}[algorithm]{Example}

\newtheoremstyle{claimstyle}{5pt}{5pt}{\em}{5pt}{\em}{:}{5pt}{}
\theoremstyle{claimstyle}

\newtheoremstyle{stepstyle}{10pt}{5pt}{\em}{0pt}{\em}{:}{5pt}{}
\theoremstyle{stepstyle}

\numberwithin{equation}{section}

\definecolor{darkblue}{rgb}{0,0,0.5}

\def\noproof{{\unskip\nobreak\hfill\penalty50\hskip2em\hbox{}\nobreak\hfill%
       $\square$\parfillskip=0pt\finalhyphendemerits=0\par}\goodbreak}
\def\endproof{\noproof\bigskip}

\newdimen\margin
\def\textno#1&#2\par{
   \margin=\hsize
   \advance\margin by -4\parindent
          \setbox1=\hbox{\sl#1}
   \ifdim\wd1 < \margin
      $$\box1\eqno#2$$
   \else
      \bigbreak
      \hbox to \hsize{\indent$\vcenter{\advance\hsize by -3\parindent
      \it\noindent#1}\hfil#2$}
      \bigbreak
   \fi}

\def\lateproof#1{\removelastskip\penalty55\medskip\noindent\setcounter{claim}{0}\setcounter{step}{0}{\bf Proof of #1. }} 

\DeclarePairedDelimiter{\abs}{\lvert}{\rvert}
\DeclarePairedDelimiter{\floor}{\lfloor}{\rfloor}
\DeclarePairedDelimiter{\ceil}{\lceil}{\rceil}

\begin{document}

\makeatletter
\def\shortrightharpoonup{{\mathchoice%
		{\clipbox{6.6mm 0.6mm 0mm 0.2mm}{$ \displaystyle\mathord-\rightharpoonup $}}%
		{\clipbox{3.55mm 0.6mm 0mm 0.2mm}{$ \textstyle\mathord-\rightharpoonup $}}%
		{\clipbox{4mm 0.6mm 0mm 0.2mm}{$ \scriptstyle\mathord-\rightharpoonup $}}%
		{\clipbox{3.75mm 0.6mm 0mm 0.2mm}{$ \scriptscriptstyle\mathord-\rightharpoonup $}}%
	}%
}

\def\shortminus{{\mathchoice%
		{\clipbox{1mm 0.6mm 0mm 0.2mm}{$ \displaystyle\mathord- $}}
		{\clipbox{1mm 0.6mm 0mm 0.2mm}{$ \textstyle\mathord- $}}
		{\clipbox{1mm 0.6mm 0mm 0.2mm}{$ \scriptstyle\mathord- $}}
		{\clipbox{1mm 0.6mm 0mm 0.2mm}{$ \scriptscriptstyle\mathord- $}}
}}
\newcommand{\overrightharpoonup}[1]{\mathchoice%
	{\vbox{\m@th\ialign{##\crcr{\normalsize$\scriptstyle\m@th \mkern-0mu\smash \shortminus\mkern-2mu\cleaders\hbox {$\mkern-2mu \scriptstyle\shortminus$}\hfill\mkern-7mu\shortrightharpoonup\mkern-0mu$}\crcr\noalign{\kern-0pt\nointerlineskip}$\hfil\displaystyle{#1}\hfil$\crcr}}}%
	{\vbox{\m@th\ialign{##\crcr{\normalsize$\scriptstyle\m@th \mkern-0mu\smash \shortminus\mkern-2mu\cleaders\hbox {$\mkern-2mu \scriptstyle\shortminus$}\hfill\mkern-7mu\shortrightharpoonup\mkern-0mu$}\crcr\noalign{\kern-0pt\nointerlineskip}$\hfil\textstyle{#1}\hfil$\crcr}}}%
	{\vbox{\m@th\ialign{##\crcr{\normalsize$\scriptscriptstyle\m@th \mkern-0mu\smash \shortminus\mkern-2mu\cleaders\hbox {$\mkern-2mu \scriptscriptstyle\shortminus$}\hfill\mkern-7mu\shortrightharpoonup\mkern-0mu$}\crcr\noalign{\kern-0pt\nointerlineskip}$\hfil\scriptstyle{#1}\hfil$\crcr}}}%
	{\vbox{\m@th\ialign{##\crcr{\normalsize$\scriptscriptstyle\m@th \mkern0mu\smash \shortminus\mkern-2mu\cleaders\hbox {$\mkern-2mu \scriptscriptstyle\shortminus$}\hfill\mkern-7mu\shortrightharpoonup\mkern-0mu$}\crcr\noalign{\kern-0pt\nointerlineskip}$\hfil\scriptscriptstyle{#1}\hfil$\crcr}}}%
}
\makeatother

\newcommand{\new}[1]{\textcolor{red}{#1}}
\newcommand{\old}[1]{\textcolor{red}{\st{#1}}}
\def\COMMENT#1{}
\def\TASK#1{}

\newcommand{\todo}[1]{\begin{center}\color{red}\textbf{to do:} #1 \end{center}}

\def\eps{{\varepsilon}}
\def\heps{{\hat{\varepsilon}}}

\newcommand{\ex}{\mathbb{E}}
\newcommand{\pr}{\mathbb{P}}
\newcommand{\cB}{\mathcal{B}}
\newcommand{\cA}{\mathcal{A}}
\newcommand{\cE}{\mathcal{E}}
\newcommand{\cS}{\mathcal{S}}
\newcommand{\cF}{\mathcal{F}}
\newcommand{\cG}{\mathcal{G}}
\newcommand{\bL}{\mathbb{L}}
\newcommand{\bF}{\mathbb{F}}
\newcommand{\bZ}{\mathbb{Z}}
\newcommand{\cH}{\mathcal{H}}
\newcommand{\cC}{\mathcal{C}}
\newcommand{\cM}{\mathcal{M}}
\newcommand{\bN}{\mathbb{N}}
\newcommand{\bR}{\mathbb{R}}
\def\O{\mathcal{O}}
\newcommand{\cP}{\mathcal{P}}
\newcommand{\cQ}{\mathcal{Q}}
\newcommand{\cR}{\mathcal{R}}
\newcommand{\cJ}{\mathcal{J}}
\newcommand{\cL}{\mathcal{L}}
\newcommand{\cK}{\mathcal{K}}
\newcommand{\cD}{\mathcal{D}}
\newcommand{\cI}{\mathcal{I}}
\newcommand{\cN}{\mathcal{N}}
\newcommand{\cV}{\mathcal{V}}
\newcommand{\cT}{\mathcal{T}}
\newcommand{\cU}{\mathcal{U}}
\newcommand{\cX}{\mathcal{X}}
\newcommand{\cZ}{\mathcal{Z}}
\newcommand{\cW}{\mathcal{W}}
\newcommand{\sT}{\mathscr{T}}
\newcommand{\1}{{\bf 1}_{n\not\equiv \delta}}
\newcommand{\eul}{{\rm e}}
\newcommand{\Erd}{Erd\H{o}s}
\newcommand{\cupdot}{\mathbin{\mathaccent\cdot\cup}}
\newcommand{\whp}{whp }
\newcommand{\bX}{\mathcal{X}}
\newcommand{\bV}{\mathcal{V}}
\newcommand{\hbX}{\widehat{\mathcal{X}}}
\newcommand{\hbV}{\widehat{\mathcal{V}}}
\newcommand{\hX}{\widehat{X}}
\newcommand{\hV}{\widehat{V}}
\newcommand{\tX}{\widetilde{X}}
\newcommand{\tV}{\widetilde{V}}
\newcommand{\cbI}{\overline{\mathcal{I}^\alpha}}
\newcommand{\hAj}{\widehat{A}^\sigma_j}
\newcommand{\hVM}{V^M}
\newcommand{\supp}{{\rm supp}}
\newcommand{\ordered}[1]{\overrightharpoonup{\mathbf{#1}}}
\newcommand{\orderededge}[1]{\overrightharpoonup{#1}}
\newcommand{\rotated}[2]{\overrightharpoonup{\mathbf{#1}}^{\langle#2\rangle}}
\newcommand{\rotatededge}[2]{\overrightharpoonup{#1}^{\langle#2\rangle}}
\newcommand{\ordsubs}[2]{(#1)_{#2}}
\newcommand{\unord}[1]{\mathbf{#1}}
\newcommand{\unordedge}[1]{#1}
\newcommand{\unordsubs}[2]{\binom{#1}{#2}}
\newcommand{\diff}{\mathop{}\!\mathrm{d}}

\newcommand{\doublesquig}{%
  \mathrel{%
    \vcenter{\offinterlineskip
      \ialign{##\cr$\rightsquigarrow$\cr\noalign{\kern-1.5pt}$\rightsquigarrow$\cr}%
    }%
  }%
}

\newcommand{\defn}{\emph}

\newcommand\restrict[1]{\raisebox{-.5ex}{$|$}_{#1}}

\newcommand{\cprob}[2]{\prob{#1 \;\middle|\; #2}}
\newcommand{\prob}[1]{\mathrm{\mathbb{P}}\left[#1\right]}
\newcommand{\expn}[1]{\mathrm{\mathbb{E}}\left[#1\right]}
\def\gnp{G_{n,p}}
\def\G{\mathcal{G}}
\def\lflr{\left\lfloor}
\def\rflr{\right\rfloor}
\def\lcl{\left\lceil}
\def\rcl{\right\rceil}

\newcommand{\qbinom}[2]{\binom{#1}{#2}_{\!q}}
\newcommand{\binomdim}[2]{\binom{#1}{#2}_{\!\dim}}

\newcommand{\grass}{\mathrm{Gr}}

\newcommand{\brackets}[1]{\left(#1\right)}
\def\sm{\setminus}
\newcommand{\Set}[1]{\{#1\}}
\newcommand{\set}[2]{\{#1\,:\;#2\}}
\newcommand{\krq}[2]{K^{(#1)}_{#2}}
\newcommand{\ind}[1]{$\mathbf{S}(#1)$}
\newcommand{\indcov}[1]{$(\#)_{#1}$}
\def\In{\subseteq}
\newcommand{\IND}{\mathbbm{1}}
\newcommand{\norm}[1]{\|#1\|}
\newcommand{\normv}[1]{\|#1\|_v}
\newcommand{\normV}[2]{\|#1\|_{#2}}

\newcommand{\Bij}{{\rm Bij}}
\newcommand{\inj}{{\rm inj}}
\newcommand{\mad}{{\rm mad}}

\newcommand{\commonneighborhood}[1]{#1-common neighborhood}

\newcommand{\bOne}{\mathbbm{1}}

\begin{abstract} 
\noindent
We prove that for any integer $k\geq 2$ and $\eps>0$,
there is an integer $\ell_0\geq 1$
such that any $k$-uniform hypergraph on $n$ vertices with minimum codegree at least $(1/2+\eps)n$ has a fractional decomposition into tight cycles of length $\ell$ ($\ell$-cycles for short)
whenever $\ell\geq \ell_0$ and~$n$ is large in terms of $\ell$.
This is essentially tight.

This immediately yields also approximate integral decompositions for these hypergraphs into $\ell$-cycles.
Moreover, for graphs this even guarantees integral decompositions into $\ell$-cycles and solves a problem posed by Glock, K\"uhn and Osthus.
For our proof, we introduce a new method for finding a set of $\ell$-cycles such that every edge is contained in roughly the same number of $\ell$-cycles from this set by exploiting that certain Markov chains are rapidly mixing.
\end{abstract}

\maketitle
\section{Introduction}
The results in this paper are motivated by questions triggered from the famous conjecture of Nash-Williams, which remains unsolved despite much research activity in the area.
Let us call a graph \emph{triangle-divisible} if all its vertices have even degree and the number of edges is divisible by $3$ (clearly necessary conditions to have a decomposition of the edge set into triangles).

\begin{conj}[Nash-Williams]\label{conj:NW}
There exists an~$n_0$ such that every triangle-divisible graph~$G$ on $n$ vertices with~$n\geq n_0$ and $\delta(G)\geq 3n/4$ admits a decomposition of its edge set into triangles.
\end{conj}

Although Conjecture~\ref{conj:NW} remains open, we already have some understanding regarding decompositions of the edge set into triangles (triangle decompositions).
Arguably, the biggest step was done by  Barber, K\"uhn, Lo and Osthus~\cite{BKLO:16}
by proving that the problem can be relaxed to fractional triangle decompositions if we increase the lower bound on the minimum degree by a minor order term.
Delcourt and Postle~\cite{DP:21} recently showed the existence of a fractional triangle decomposition in a graph~$G$ on $n$ vertices whenever $\delta(G)\geq 0.828n$ and $n$ is large.
In turn this implies with the result above that every large triangle-divisible  graph $G$ on $n$ vertices with $\delta(G)\geq 0.83n$ admits a triangle decomposition.

We say a hypergraph is \emph{$k$-uniform} if all its edges contain exactly $k$ vertices and we also refer to $k$-uniform hypergraphs as \emph{$k$-graphs}; 2-graphs are as usual called graphs.
Given a $k$-graph $F$, an \emph{$F$-decomposition} of a $k$-graph $H$ is a decomposition of the edge set of $H$ into edge-disjoint copies of $F$.
The topic of $k$-graph decompositions has experienced major breakthroughs in the last decade; nevertheless many problems remain unsolved~\cite{GKLO:16,keevash:14,keevash:18b}.

Let the \emph{degree} $d(\unord{x})$ of a $(k-1)$-set $\unord{x}$ of vertices in a $k$-graph $H$ be the number of edges that contain $\unord{x}$ (observe that this coincides with the usual degree for graphs).
Let $\delta(H)$ be the \emph{minimum degree} of $H$ defined by $\min\bigl\{d(\unord{x})\colon \unord{x}\in \binom{V(H)}{k-1}\bigr\}$.
Note that in order to guarantee an $F$-decomposition, $H$ has to satisfy certain divisibility conditions, including $e(F)|e(H)$;
for graphs, we also need that $gcd(d_F(v)\colon v\in V(F))$ divides $gcd(d_H(v)\colon v\in V(H))$ and whenever $H$ satisfies these conditions, we say that $H$ is \emph{$F$-divisible}.
For $k\geq 3$, the concept of $F$-divisibility is similar but more complex and we refer the reader to~\cite{GKLO:16} for (many) more details.

In this paper, we consider Dirac-type (minimum degree) conditions for $k$-graphs that enforce $F$-decompositions.
In particular, we are interested in the \emph{$F$-decomposition threshold} $\delta_F$ which is defined as the least number $\delta$
such that for every $\mu>0$, there exists an integer $n_0$ such that every $F$-divisible $k$-graph $H$ on $n\geq n_0$ vertices
with $\delta(H)\geq (\delta+\mu)n$ admits an $F$-decomposition.

The $F$-decomposition threshold is very closely linked to the threshold for approximate decompositions and fractional decompositions, respectively.
For $\eta\geq 0$, we say that an \emph{$\eta$-approximate $F$-decomposition} of $H$ is an edge-disjoint collection of copies of $F$ that cover all but at most $\eta n^k$ edges of $H$.
Let $\delta_F^\eta$ be the least number~$\delta$
such that for every $\mu>0$, there exists an integer $n_0$ such that every $k$-graph $H$ on $n\geq n_0$ vertices
with $\delta(H)\geq (\delta+\mu)n$ admits an $\eta$-approximate $F$-decomposition.
Clearly, $\delta_F^\eta$ is monotonically increasing as $\eta$ tends to $0$.
Let $\delta_F^{0+}:=\sup_{\eta>0}\delta_F^\eta$ be the \emph{approximate $F$-decomposition threshold}.

Let $\cF(H)$ be the set of all copies of $F$ in $H$.
We say a function $\omega\colon \cF(H)\to [0,1]$ with $\sum_{F\in \cF(H)\colon e\in E(F)} \omega(F)=1$ for all $e\in E(H)$ is a \emph{fractional $F$-decomposition}
and we define the \emph{fractional~$F$-decomposition threshold $\delta_F^*$} as the least number $\delta$ such that for every $\mu>0$, there exists an integer $n_0$ such that every $k$-graph $H$ on $n\geq n_0$ vertices with $\delta(H)\geq (\delta+\mu)n$ admits a fractional $F$-decomposition.

Haxell and R\"odl \cite{HR:01} showed that $\delta_F^*\geq \delta_F^{0+}$ for graphs and later this was extended to $k$-graphs for arbitrary $k$ in~\cite{RSST:07} using the (hypergraph) regularity lemma.
Therefore\footnote{For the first inequality, we ignore the fact that for integral decompositions we in addition assume that $H$ is $F$-divisible. To avoid this, we could define the fractional/approximate decomposition threshold also only for $F$-divisible $k$-graphs. Since this technicality plays no role in this paper, we stick with the (natural) definitions as above.}, $\delta_F \geq \delta_F^* \geq \delta_F^{0+}$.
However, crucially for many $F$ it is known that all three values coincide.
This reduces the complexity of finding $F$-decompositions.

Let us discuss now some results that relate these different decomposition thresholds.

\begin{theorem}[Barber, K\"uhn, Lo and Osthus~\cite{BKLO:16}] \label{thm:BKLO}
	Let $F$ be an $r$-regular graph.
	Then $\delta_F\leq \max\{\delta_F^{0+},1-1/(3r)\}$.
	Moreover,
	$\delta_{C_4}=2/3$ and $\delta_{C_\ell}=1/2$ for even $\ell\geq 6$.
	For odd $\ell$, we have $\delta_{C_\ell}=\delta_{C_\ell}^{0+}$.
\end{theorem}

Further significant progress on determining $\delta_F$ is due to Glock, K\"uhn, Lo, Montgomery and Osthus~\cite{GKLMO:19}.
They calculated the decomposition threshold for all bipartite graphs $F$; in particular, their results imply $\delta_F\in \{0,1/2,2/3\}$ in this case.
For graphs $F$ with chromatic number $\chi\geq 5$, they proved that $\delta_F\in \{\delta_F^*,1-1/\chi,1-1/(\chi+1)\}$.

When considering $\delta_F$ for $k$-graphs $F$ with $k\geq 3$ we are immediately in deep waters.
Only recently Keevash~\cite{keevash:14} verified the existence conjecture, that is, he showed that whenever the complete $k$-graph on $n$ vertices $K_n^{k}$ is $K_r^{k}$-divisible and $n$ is large enough, then $K_n^{k}$ has a $K_r^{k}$-decomposition.
Glock, K\"uhn, Lo and Osthus greatly improved this by showing that $K_r^{k}$ can be replaced by any other $k$-graph~$F$~\cite{GKLO:16}.
In addition, they worked in a framework that yields bounds on $\delta_F$ that can be explicitly calculated.

\medskip

Our contribution concerns (tight) cycles.\footnote{Observe that decompositions into tight cycle directly yield decompositions into other types of $k$-graph cycles.}
A \emph{cycle $C_\ell^{k}$ of length $\ell$} is a $k$-graph on $\ell$ vertices whose vertex set can be cyclically ordered such that a $k$-set forms an edge of~$C_\ell^{k}$ if and only if its elements appear consecutively in this ordering (when $k=2$, we also write $C_\ell$ instead of $C_\ell^2$).
Our main result is as follows.

\begin{theorem}\label{thm: main result}
	For every integer $k\geq 2$ and $\eps>0$,
	there exists an $\ell_0$ such that $\delta_{C_\ell^{k}}^*\leq 1/2+\eps$ for all $\ell\geq \ell_0$.
\end{theorem}

Note that for graphs, by Theorem~\ref{thm:BKLO} and as $\delta_F \geq \delta_F^* \geq \delta_F^{0+}$, we can replace $\delta_{C_\ell}^*$ by $\delta_{C_\ell}$ in Theorem~\ref{thm: main result}.
Our main result also answers a question (in a strong form) by Glock, K\"uhn and Osthus who asked whether the statement of the theorem holds for $k=2$. 

As we mentioned above, for $k\geq 3$ decomposition problems are much more complicated than for graphs and 
thus is it not surprising that beside the fact that $\delta_F^{0+}=0$ whenever $F$ is $k$-partite almost nothing is known about the exact values of $\delta_F^{0+},\delta_F^*,\delta_F$.
Our main result implies almost tight bounds on $\delta_{C_\ell^k}^*$ and $\delta_{C_\ell^k}^{0+}$ 
because whenever $\ell$ is not divisible by $k$, we have $\delta_{C_\ell^k}^{0+}\geq \frac{1}{2}+\frac{1}{(k-1+2/k)(\ell-1)}$ (see Example~\ref{example} in Section~\ref{sec:2}).

In view of further applications we prove a more general statement than in Theorem~\ref{thm: main result}.
In many scenarios,
we would like to find a (fractional/approximate) $F$-decomposition of a typical random thin edge slice of a $k$-graph that admits such an $F$-decomposition.
Hence we need an assumption on the $k$-graph that is robust with respect to taking random edge slices and covers the minimum degree condition in Theorem~\ref{thm: main result}.
A property which has this ability is  $(\alpha,\ell)$-connectedness, which we introduce below.

In addition, our fractional decompositions have the property that all weights on the $\ell$-cycles are equal up to constant factors.
This is another useful property for applications.

For a $k$-graph $H$ with vertex set~$V$ and edge set~$E$ on $n$ vertices, 
we write~$\orderededge{E}(H)$ for the set of tuples~$(v_1,\ldots,v_k)\in V^k$ such that~$\{v_1,\ldots,v_k\}\in E$ and we define~$\orderededge{e}(H):=\abs{\orderededge{E}(H)}$.
A \defn{walk}~$W$ of length $\ell$ in $H$ is a sequence $\orderededge{e}_1\ldots \orderededge{e}_{\ell}$ of elements of~$\orderededge{E}(H)$ such that there is a sequence $v_1\ldots v_{\ell+(k-1)}$ of vertices of $H$ with $\orderededge{e}_i=(v_i,\ldots,v_{i+(k-1)})$ for all $i\in[\ell]$.
We call~$W$ a walk from~$\orderededge{e}_1$ to~$\orderededge{e}_\ell$.
We say that $H$ is \defn{$(\alpha,\ell)$-connected} if for all~$\orderededge{e},\orderededge{f}\in \orderededge{E}(H)$, 
there are at least~$\alpha n^{\ell-1}/\orderededge{e}(H)$ walks from~$\orderededge{e}$ to~$\orderededge{f}$ in~$H$ that have length~$\ell$. 

\begin{theorem}\label{thm: stronger version}
	For all~$\alpha\in(0,1)$,~$\mu\in(0,1/3)$ and integers~$\ell,k\geq 2$, there is an~$n_0$ such that the following holds for all~$n\geq n_0$. Suppose~$H$ is an~$(\alpha,\ell_0)$-connected~$k$-graph on~$n$ vertices for some $\ell_0\geq k+1$ with~$ 180k\frac{\ell_0}{\alpha}\log\frac{\ell_0}{\alpha}\log\frac{1}{\mu}\leq \ell$. Then there is a fractional~$C^k_{\ell}$-decomposition~$\omega$ of~$H$ with
	\begin{equation*}
		(1-\mu)\frac{2e(H)}{\Delta(H)^{\ell}}\leq\omega(C)\leq (1+\mu)\frac{2e(H)}{\delta(H)^{\ell}}
	\end{equation*}
	for all $\ell$-cycles~$C$ in~$H$.
\end{theorem}
The fact that Theorem~\ref{thm: stronger version} indeed implies Theorem~\ref{thm: main result} is an immediate consequence of Lemma~\ref{remark: connecting} below
as it implies that $\delta(H)\geq \frac{1+\alpha}{2}n$ yields $(\alpha^{3k!},k^2)$-connectedness.

In another article together with Sch\"ulke,
we employ Theorem~\ref{thm: stronger version} to prove a strong generalization of the well-known result due to R\"odl, Ruci\'nski and Szemer\'edi~\cite{RRS:08}.
We show that every $k$-graph~$H$ on~$n$ vertices with minimum degree $\delta(H)\geq (1/2+o(1))n$ not only contains one Hamilton cycle but essentially as many edge-disjoint Hamilton cycles as $H$ may potentially have,
namely, the largest $p$ for which $H$ has a spanning subgraph where every vertex is contained in~$kp$ edges.

\medskip

Next we give an overview of the proof of Theorem~\ref{thm: stronger version}. 
When dealing with fractional decompositions many approaches have considered so far
the strategy to start with the (appropriately scaled) uniform distribution on $\cF(H)$ and then turn this into a fractional decomposition by shifting some weight around the graph.
In this case, the crucial question is in how many copies of $F$ does a particular edge~$e$ lie because the total weight on~$e$ will be proportional to this quantity.
(From now on $F$ is a cycle of length $\ell$.)
However, when dealing with $k$-graphs of minimum degree close to $n/2$ this may vary significantly and starting with the uniform distribution on $\cF(H)$ seems hopeless to us.

Our main contribution is a novel approach to overcome this problem.
In Section~\ref{sec:initial weight}, we first find a set of $\ell$-cycles $\cC\subseteq \cF(H)$ such that every edge of~$H$ lies in roughly the same number of elements of $\cC$.
Consequently, a scaled uniform distribution on~$\cC$ almost yields a fractional~$F$-decomposition.
In a final step (see Section~\ref{sec:transporter}), we adjust the weight distribution slightly and find the desired fractional decomposition.
So how do we find $\cC$?
This can be done as follows. 
Fix an ordered $(k-1)$-set $\unord{x}$.
Among all edges that contain $\unord{x}$, we introduce certain restrictions; that is, for some $u,v$ that form each together with $\unord{x}$ an edge, 
we forbid that a cycle contains the vertices $u,\unord{x},v$ (or vice versa) in this order. We proceed similarly for each~$\unord{x}\in\unordsubs{V(H)}{k-1}$.
Let~$\cC$ be the set of all cycles that respect all these imposed restrictions.

The restrictions will be constructed in such way that for each $u,\unord{x}$, there is always the same number of choices $r$, say, for the next vertex of the cycle.
Consequently, starting with a $(k-1)$-set $\unord{x}$, there are almost exactly $r^m$ choices (we have to avoid repetitions) for the subsequent $m$ vertices.

Construct an auxiliary graph $J$ as follows; the vertex set of $J$ is the set of all ordered edges of $H$
and two vertices $\orderededge{e},\orderededge{f}$ are adjacent if there is an ordering of $k+1$ vertices of $H$ such that~$\orderededge{e}$ coincides with the first $k$ vertices and $\orderededge{f}$ with the last $k$ vertices and our restrictions allow to transition from $\orderededge{e}$ to $\orderededge{f}$.

Now consider a simple random walk on $J$.
Then the $m$-step transition probability from $\orderededge{e}$ to $\orderededge{f}$ multiplied by $r^m$ essentially equals the number of paths from $\orderededge{e}$ to $\orderededge{f}$.
One of the key steps in our argument is that this random walk mixes rapidly. 
From this it is not hard to deduce that every edge of $H$ is contained in roughly the same number of cycles in $\cC$.

Let us also mention that we avoid the application of the arguably very technical hypergraph regularity lemma, which usually is applied when considering decomposition questions, and hence our proofs are comparably short and without much clutter.

\section{Preliminaries}\label{sec:2}
\subsection{Notation} 
For an integer~$n\geq 1$, we define \defn{$[n]:=\{1,\ldots,n\}$} and \defn{$[n]_0:=[n]\cup \{0\}$}.
For a set $A$, we say that $A$ is a \defn{$k$-set} if $\abs{A}=k$ and 
we define \defn{$\unordsubs{A}{k}:=\{ A'\subseteq A\colon \abs{A'}=k \}$}.
We often use $\unord{x},\unord{y}$ to refer to sets and $\ordered{x},\ordered{y}$ when considering tuples; however, if the tuple arises from ordering the vertices of an edge, then we use $\orderededge{e},\orderededge{f},\orderededge{g},\orderededge{h},\orderededge{s},\orderededge{t}$.
We may subsequently drop the arrow to denote the underlying set of an ordered set, so that for an ordered set~$\ordered{x}=(x_1,\ldots,x_k)$, we have~$\unord{x}=\{x_1,\ldots,x_k\}$.
We generally identify a sequence $x_1\ldots x_n$ with the tuple $(x_1,\ldots,x_n)$.
An \defn{ordering} of a $k$-set $\unord{x}=\{x_1,\ldots,x_k\}$ is a sequence $x_1\ldots x_k$ without repetitions.  

For real numbers $\alpha,\beta,\delta,\delta'$, we write \defn{$\alpha=(1\pm \delta)\beta$} whenever $(1-\delta)\beta\leq \alpha\leq (1+\delta)\beta$ and \defn{$(1\pm \delta)\alpha=(1\pm \delta')\beta$} whenever $(1-\delta')\beta\leq (1- \delta)\alpha\leq (1+\delta)\alpha\leq(1+\delta')\beta$.
We write \defn{$\alpha\ll\beta$} to mean that there is a non-decreasing function~$\alpha_0\colon(0,1]\rightarrow(0,1]$ such that for any $\beta\in(0,1]$, the subsequent statement holds for~$\alpha\in(0,\alpha_0(\beta)]$.
Hierarchies with more constants are defined similarly and should be read from right to left.
Constants in hierarchies will always be real numbers in $(0,1]$.
Moreover, if $1/n$ appears in a hierarchy, this implicitly means that $n$ is a natural number.
We ignore rounding issues when they do not affect the argument.

Let $H$ be a $k$-graph on $n$ vertices, where~$k\geq 2$. Whenever we use~$k$ to refer to the uniformity of a hypergraph we tacitly assume that~$k\geq 2$.
We write \defn{$V(H)$} for the vertex set and \defn{$\unordedge{E}(H)$} for the edge set of $H$.

A walk $W$ of length $\ell$, or simply an $\ell$-\defn{walk}, in $H$ is a sequence $\orderededge{e}_1\ldots \orderededge{e}_{\ell}$ of elements of $\orderededge{E}(H)$ such that there is a sequence $v_1\ldots v_{\ell+(k-1)}$ of vertices of $H$ with $\orderededge{e}_i=(v_i,\ldots,v_{i+(k-1)})$ for all $i\in[\ell]$; we call this sequence of vertices the \defn{vertex sequence} of $W$ and define \defn{$V(W):=\{v_1,\ldots, v_{\ell+(k-1)}\}$} and~\defn{$E(W):=\{e_{i}\colon i\in[\ell]\}$}.
We say that $W$ is \defn{self-avoiding} if $v_i\neq v_j$ for all distinct $i,j\in[\ell]$ and we say that~$W$ is \defn{internally self-avoiding} if~$e_1\ldots e_{\ell-k}$ and~$e_{k+1}\ldots e_{\ell}$ are self-avoiding.
We call~$W$ \defn{closed} if~$\orderededge{e}_1=\orderededge{e}_{\ell}$.
We say a cycle~$C$ is given by an internally self-avoiding closed walk~$\orderededge{e}_1\ldots \orderededge{e}_{\ell}$ if~$E(C)=\{e_1,\ldots, e_{\ell}\}$.
We write~\defn{$\cC_{\ell}(H)$} for the set of~$\ell$-cycles in~$H$.

\subsection{Chernoff's inequality} 
We will also need the following standard concentration inequality. 
\begin{lemma}\label{lemma: Chernoff}
	Let $ X_1,\ldots,X_n$ be independent Bernoulli random variables with parameter $p$ and let~$X:=\sum_{i\in[n]} X_i$.
	Then for any $ \delta\in(0,1]$,
	\begin{equation*}
		\pr[|X-\ex[X]| \geq  \delta\ex[X]]\leq 2\exp\biggl( -\frac{\delta^2}{3}\ex[X] \biggr).
	\end{equation*}
\end{lemma}

\subsection{A lower bound on the (fractional) cycle decomposition threshold}

The following example shows that the bound in Theorem~\ref{thm: main result} cannot be improved by much.

\begin{example}\label{example}
	Let $1/n\ll \zeta \ll \delta\ll 1/k, 1/\ell$, where~$\ell\not\equiv 0\mod k$ and~$n$ is even.
	Suppose~$\eps=\frac{1}{(k-1+2/k)(\ell-1)}-\delta$, let~$A$ and~$B$ be disjoint sets of size $n/2$, define $V:=A\cup B$ and for~$i\in[k]_0$, let~$E_i$ be the set of edges~$e$ of the complete~$k$-graph on~$V$ with $\abs{e\cap A}=i$. 
	Suppose every edge in~$E_0\cup E_2$ is selected independently at random with probability~$2(1+\zeta)\eps$. 
	By Lemma~\ref{lemma: Chernoff}, with positive probability, the set of selected edges is the edge set of a $k$-graph~$H_{0,2}$ with vertex set $V$ and $d(\unord{x})\geq \eps n +k$ for all $(k-1)$-sets $\unord{x}\subseteq V$ with $\abs{\unord{x}\cap A}\leq 2$ as well as\COMMENT{%
		Note: $2(1+2\zeta)\eps\left(\binom{n/2}{k}+\binom{n/2}{2}\binom{n/2}{k-2}\right)
		\leq 2(1+2\zeta)\eps\left(\frac{n^k}{2^k k!}+\frac{n^k}{2^{k+1}(k-2)!}\right)
		= (1+2\zeta)\eps\left( \frac{2}{k}+k-1 \right)\frac{n^k}{2^k (k-1)!}
		=(1+2\zeta)\left(\frac{1}{\ell-1}-\delta\left( \frac{2}{k}+k-1 \right)\right)\frac{n^k}{2^k (k-1)!}
		\leq (1+2\zeta)\frac{1}{\ell-1}\frac{n^k}{2^k (k-1)!}-\delta\frac{n^k}{2^k (k-1)!}
		\leq (1+4\zeta)\frac{1}{\ell-1}\binom{n/2}{1}(1-\zeta)\frac{n^{k-1}}{2^{k-1} (k-1)!}-\delta^{3/2}n^k
		\leq (1+4\zeta)\frac{1}{\ell-1}\binom{n/2}{1}\binom{n/2}{k-1}-\delta^{3/2}n^k
		=(1+4\zeta)\frac{\abs{E_1}}{\ell-1}-\delta^{3/2}n^k
		\leq\frac{\abs{E_1}}{\ell-1}-\delta^{3/2}n^k+4\zeta n^k
		\leq\frac{\abs{E_1}}{\ell-1}-\delta^{3/2}n^k/2
		\leq\frac{\abs{E_1}}{\ell-1}-\delta^{2}n^k
		\leq\frac{\abs{E_1}-\delta^{2}n^k}{\ell-1} $
	}
	\begin{align*}
		e(H_{0,2})
		&\leq 2(1+2\zeta)\eps \left(\binom{n/2}{k}+\binom{n/2}{2}\binom{n/2}{k-2}\right)\\
		&\leq (1+2\zeta)\eps\left(k-1+\frac{2}{k}\right)\frac{n^k}{2^{k}(k-1)!}
		< \frac{\abs{E_1}- \delta^2n^k}{\ell-1}.
	\end{align*}
	Let~$H_{0,2}$ be such a~$k$-graph and let~$H:=(V,E(H_{0,2})\cup\bigcup_{i\in[k]\setminus\{2\}}E_i)$.
	Then $\delta(H)\geq \left(\frac{1}{2}+\eps\right)n$ holds.
	
	Note the following: Suppose~$v_1\ldots v_{\ell'}$ is the vertex sequence of a walk in~$(V,E_1)$, then~$v_i\in B$ implies~$v_{j}\notin B$ if~$i<j\leq i+k-1$ as well as~$v_{j}\in B$ if~$j=i+k$. Thus, since~$\ell\not\equiv 0\mod k$, every~$\ell$-cycle in $H$ that has an edge in $E_1$ must have at least one edge that is an edge of $H_{0,2}$. For every function $\omega\colon \cC_{\ell}(H)\rightarrow[0,1]$ with $\sum_{C\in \cC_{\ell}(H)\colon e\in E(C)}\omega(C)\leq 1$ for all $e\in E(H)$, this implies
	\begin{align*}
		\sum_{e\in E_1}\sum_{\substack{C\in \cC_{\ell}(H)\colon\\e\in E(C)}}\omega(C)
		&\leq (\ell-1) \sum_{\substack{C\in \cC_{\ell}(H)\colon\\E(C)\cap E_1}}\omega(C)\\
		&\leq (\ell-1) \sum_{e\in E(H_{0,2})}\sum_{\substack{C\in \cC_{\ell}(H)\colon\\e\in E(C)}}\omega(C)\\
		&\leq (\ell-1) e(H_{0,2})\\
		&<\abs{E_1}- \delta^2 n^k
	\end{align*}
	and thus $H$ admits neither a fractional nor a $\delta^2$-approximate $C_{\ell}^k$-decomposition.
\end{example}

\subsection{Intersecting hypergraphs}

For~$\unord{x}\in\unordsubs{V(H)}{k-1}$, we define~\defn{$N_H(\unord{x}):=\{ v\in V(H):\unord{x}\cup\{v\}\in E(H) \}$}.
Let us call a~$k$-graph~$H$ on~$n$ vertices~\defn{$\alpha$-intersecting} if~$\abs{N_H(\unord{x})\cap N_H(\unord{y})}\geq \alpha n$ holds for all~$\unord{x},\unord{y}\in\unordsubs{V(H)}{k-1}$; in particular~$H$ is~$\alpha$-intersecting if~$\delta(H)\geq\frac{1+\alpha}{2} n$.
Theorem~\ref{thm: stronger version} applies to $(\alpha,\ell_0)$-connected $k$-graphs, a property which is sometimes slightly complicated to verify.\COMMENT{Check, $\ell$ or $\ell_0$.}
We show in Lemma~\ref{remark: connecting} that all $\alpha$-intersecting $k$-graphs are $(\alpha^{3k!},k^2-k+2)$-connected.
The statement is in fact already proved by R\"odl, Ruci\'nski and Szemer\'edi in~\cite{RRS:08} with a weaker dependence of the parameters and a fairly long proof.
The idea of the short proof below is due to Reiher (unpublished)~\cite{Rei:20} and we give it here to make it publicly available.


\begin{lemma}\label{remark: connecting}
	Let $1/n\ll\alpha,1/k$. Suppose~$H$ is an~$\alpha$-intersecting~$k$-graph on~$n$ vertices. Then~$H$ is~$(\alpha^{3k!},k^2-k+2)$-connected.
\end{lemma}
\begin{proof}
	We will use induction on~$k$ to show the following.
	\begin{align*}
		\begin{minipage}[c]{0.85\textwidth}\em
			Suppose~$a_{k}=(k-1)!\sum_{i=0}^{k-2} \frac{i+1}{i!}$,~$\ell_k=k^2-k+2$ and~$H$ is an~$\alpha$-intersecting~$k$-graph on~$n$ vertices. Then the number of~$\ell_k$-walks from~$\orderededge{s}$ to~$\orderededge{t}$ in~$H$ is at least~$\alpha^{a_k}n^{\ell_k-k-1}$ for all~$\orderededge{s},\orderededge{t}\in\orderededge{E}(H)$.
		\end{minipage}\ignorespacesafterend
	\end{align*}
	This suffices because for all~$k\geq 2$ and every~$\alpha$-intersecting~$k$-graph~$H$ on~$n$ vertices (where $n$ is sufficiently large), we have~$\orderededge{e}(H)\geq \delta(H)\prod_{i\in[k-2]_0}(n-i)\geq \alpha^2 n^k$ and~$a_k\leq k!(\sum_{i=0}^{\infty} 1/i!)-2\leq 3k!-2$\COMMENT{%
		$a_k=(k-1)!\sum_{i=0}^{k-2} \frac{i+1}{i!}=(k-1)!\biggl(\sum_{i=1}^{k-2} \frac{1}{(i-1)!}+\sum_{i=0}^{k-2} \frac{1}{i!}\biggr)=(k-1)!\biggl(\sum_{i=0}^{k-3} \frac{1}{i!}+\sum_{i=0}^{k-2} \frac{1}{i!}\biggr)\leq (k-1)!\biggl(\sum_{i=0}^{\infty} \frac{1}{i!}+\sum_{i=0}^{\infty} \frac{1}{i!}\biggr)-2\leq (k-1)!\cdot 2\eul-2\leq \frac{6}{k}k!-2\leq 3k!-2$
	}.
	With some foresight we remark that~$a_{k}=(k-1)(a_{k-1}+k-1)$ holds for all~$k\geq 3$\COMMENT{%
		$(k-1)(a_{k-1}+k-1)=(k-1)!\left(\frac{k-1}{(k-2)!}+\sum_{i=0}^{k-3}\frac{i+1}{i!}\right)=(k-1)!\sum_{i=0}^{k-2}\frac{i+1}{i!}=a_k$
	}.
	
	For~$k=2$, we have~$\ell_k=4$ and~$a_k=1$ and thus the desired statement is obviously true since~$H$ is~$\alpha$-intersecting.
	
	Suppose now that $k\geq 3$.
	Let~$V:=V(H)$ and~$(s_1,\ldots,s_k),(t_1,\ldots,t_k)\in\orderededge{E}(H)$. For $z\in V$, consider the indicator function $\bOne_z$ defined on $V^{(k-1)(k-2)}$ where $\bOne_z(\ordered{x})=1$ if and only if~$s_2\ldots s_{k}\ordered{x}t_1\ldots t_{k-1}$ is the vertex sequence of a $((k-1)^2+1)$-walk in the link~$L_z$ of $z$, that is the $(k-1)$-graph~$L_z$ with vertex set $V\setminus\{z\}$ where a $(k-1)$-set~$e\subseteq V\setminus\{z\}$ is an edge of $L_z$ if and only if $e\cup\{z\}$ is an edge of $H$.
	Let
	\begin{equation*}
		N:=N_H(\{s_2,\ldots,s_k\})\cap N_H(\{ t_1,\ldots,t_{k-1} \}).
	\end{equation*}
	Note that for all~$\ordered{x}=(x_1,\ldots,x_{(k-1)(k-2)})\in V^{(k-1)(k-2)}$, we can obtain the vertex sequence~$\orderededge{s}\ordered{y}\orderededge{t}$ of an~$\ell_k$-walk in~$H$ by inserting vertices~$z\in N$ with~$\bOne_z(\ordered{x})=1$ into~$\orderededge{s}\ordered{x}\orderededge{t}$ every $(k-1)$ vertices, that is by choosing not necessarily distinct~$z_1,\ldots,z_{k-1}\in N$ with~$\bOne_{z_i}(\ordered{x})=1$ for all~$i\in[k-1]$ and defining~$\ordered{y}=y_1\ldots y_{(k-1)^2}$ as the sequence with $y_{k(i-1)+1}=z_i$ for all $i\in[k-1]$ that has $\ordered{x}$ as a subsequence. For every~$\ordered{x}\in V^{(k-1)(k-2)}$, the number of~$\ell_k$-walks from~$\orderededge{s}$ to ~$\orderededge{t}$ in~$H$ given by such insertion constructions is $(\sum_{z\in N} \bOne_z(\ordered{x}))^{k-1}$.
	For all~$z\in V$, the link~$L_z$ of~$z$ is~$\alpha$-intersecting. Thus the number of~$(k-1)$-walks in~$L_z$ starting at~$(s_2,\ldots,s_k)$ is at least~$\alpha^{k-2} n^{k-2}$ and the induction hypothesis implies that for all tuples~$\orderededge{s}'$ such a walks ends at, the number of~$\ell_{k-1}$-walks from~$\orderededge{s}'$ to~$(t_1,\ldots,t_{k-1})$ in~$L_z$ is at least~$\alpha^{a_{k-1}}n^{\ell_{k-1}-k}$. This yields
	\begin{equation*}
		\sum_{\ordered{x}\in V^{(k-1)(k-2)}} \bOne_z(\ordered{x})\geq \alpha^{a_{k-1}+k-2}n^{\ell_{k-1}-2}=\alpha^{a_{k-1}+k-2}n^{(k-1)(k-2)}.
	\end{equation*}
	Since~$x\mapsto x^{k-1}$ is convex, we conclude with Jensen's inequality that the number of $\ell_k$-walks from $\orderededge{s}$ to $\orderededge{t}$ in $H$ is at least
	\begin{align*}
		\sum_{\ordered{x}\in V^{(k-1)(k-2)}} \Bigl(\sum_{z\in N} \bOne_z(\ordered{x})\Bigr)^{k-1}
		&\geq n^{(k-1)(k-2)}\Bigl(n^{-(k-1)(k-2)} \sum_{z\in N} \sum_{\ordered{x}\in V^{(k-1)(k-2)}} \bOne_z(\ordered{x})\Bigr)^{k-1}\\
		&\geq n^{(k-1)(k-2)}(n^{-(k-1)(k-2)} \cdot\alpha n\cdot \alpha^{a_{k-1}+k-2}n^{(k-1)(k-2)})^{k-1}\\
		&= \alpha^{a_k} n^{\ell_k-k-1},
	\end{align*}
 	which completes the proof.
\end{proof}

\section{Rough weight assignment}\label{sec:initial weight}

The key insight for the proof of Theorem~\ref{thm: stronger version} is the application of Markov chain limit theory with the purpose of estimating the number of cycles that contain a fixed edge and obey well-chosen conditions. 
In particular, we exploit that certain Markov chains are rapidly mixing; that is, the speed of convergence to the limit distribution is exponential.
In order to easily relate a Markov chain with the number of walks starting at an edge we restrict our set of admissible walks by what we call transition systems.

Suppose $H$ is a $k$-graph with vertex set $V$. 
For $\unord{x}\in\unordsubs{V}{k-1}$, a \defn{transition graph} of~$H$ at~$\unord{x}$ is a graph with vertex set $\{e\in E(H)\colon \unord{x}\In e\}$.
A \defn{transition system} of $H$ is a family $\cT=(T_{\unord{x}})_{\unord{x}\in\unordsubs{V}{k-1}}$ where~$T_{\unord{x}}$ is a transition graph of~$H$ at~$\unord{x}$.
A walk in~$H$ with vertex sequence $v_1\ldots v_{\ell+k-1}$ is \defn{$\cT$-compatible} if~$\{v_i,\ldots,v_{i+k-1}\}$ and~$\{v_{i+1},\ldots,v_{i+k}\}$ are adjacent in~$T_{\{ v_{i+1},\ldots,v_{i+k-1} \}}$ for all~$i\in[\ell-1]$ and 
if~$C$ is a cycle, then for~$C$ to be~\defn{$\cT$-compatible} we require that~$C$ is given by a~$\cT$-compatible internally self-avoiding closed walk.
We say that $\cT$ is \defn{$r$-regular} if $T_\unord{x}$ is $r$-regular for all $\unord{x}\in \binom{V}{k-1}$ and we say that $\cT$ is \defn{regular} if $\cT$ is $r$-regular for some $r$.
We will often choose $\cT$ randomly; to be more precise,
we say that $\cT$ is a \emph{random $r$-transition system} if 
for $\unord{x}\in\unordsubs{V}{k-1}$, the graph $T_{\unord{x}}$ is an $r$-regular transition graph of~$H$ at~$\unord{x}$ chosen independently and uniformly at random among all $r$-regular subgraphs of the complete graph with vertex set $\{e\in E(H)\colon \unord{x}\In e\}$.
Note that graphs with transition systems already have been studied, often with an algorithmic perspective (see for example~\cite{Sze:03}).

For a given~$(\alpha,\ell)$-connected~$k$-graph~$H$ on~$n$ vertices, the first step in the proof of Theorem~\ref{thm: stronger version} is the choice of a transition system $\cT$ such that every edge of~$H$ is contained in approximately equally many $\cT$-compatible cycles in $H$. Lemmas~\ref{lemma: transition system connecting} and~\ref{lemma: counting walks} imply that this happens with high probability when $\cT$ is a random $r$-transition system.
We prove this as follows.
Lemma~\ref{lemma: transition system connecting} verifies that $(\alpha,\ell)$-connectedness is maintained (with high probability) when moving to a random transition system provided the parameter $\alpha$ is appropriately lowered.
In Lemma~\ref{lemma: mixing}, we prove that certain Markov chains are rapidly mixing and
apply this in Lemma~\ref{lemma: counting walks} to prove the desired result about cycle counts.


We say that a~$k$-graph~$H$ on~$n$ vertices with~$r$-regular transition system~$\cT$ is~\defn{$\cT$-compatibly~$(\alpha,\ell)$-connected} if the number of~$\cT$-compatible~$\ell$-walks from~$\orderededge{s}$ to~$\orderededge{t}$ is at least~$\alpha r^{\ell-1}/\orderededge{e}(H)$ for all~$\orderededge{s},\orderededge{t}\in\orderededge{E}(H)$.

\begin{lemma}\label{lemma: transition system connecting}
	Let~$1/n\ll\alpha,\eps,1/k,1/\ell$.
	Let~$H$ be an~$(\alpha,\ell)$-connected~$k$-graph on~$n$ vertices,~$r$ an even integer with~$\eps n\leq r\leq\delta(H)$ and suppose $\cT$ is a random $r$-transition system of~$H$. Then~$H$ is~$\cT$-compatibly~$(\alpha/2,\ell)$-connected with probability at least~$1-\exp(-\sqrt{n})$.
\end{lemma}
\begin{proof}
	We will show that for sufficiently many suitably chosen disjoint sets of~$\ell$-walks in~$H$ of size~$n^{3/4}$, the number of~$\cT$-compatible walks in these sets is not much lower than their expected value with very high probability. A union bound then finishes the argument.
	
	To this end, fix tuples~$\orderededge{s}=(s_1,\ldots,s_k)$ and~$\orderededge{t}=(t_1,\ldots,t_k)$ in~$\orderededge{E}(H)$. Let~$\cW$ denote the set of the internally self-avoiding~$\ell$-walks from~$\orderededge{s}$ to~$\orderededge{t}$ in~$H$. The number of~$\ell$-walks from~$\orderededge{s}$ to~$\orderededge{t}$ in~$H$ that are not internally self-avoiding is at most~$\ell^2 n^{\ell-k-2}\leq \frac{1}{4}\alpha n^{\ell-1}/\orderededge{e}(H)$ and thus~$\abs{\cW}\geq\frac{3}{4}\alpha n^{\ell-1}/\orderededge{e}(H)$ holds. Let~$p:=\frac{2}{3}\alpha n^{\ell-7/4}/\orderededge{e}(H)$ and~$q:=n^{3/4}$. We wish to obtain pairwise disjoint~$q$-sets~$\cW_1,\ldots,\cW_{p}$ of internally self-avoiding~$\ell$-walks from~$\orderededge{s}$ to~$\orderededge{t}$ in~$H$ such that for all~$i\in[p]$ and~$W,W'\in\cW_i$, the internally self-avoiding walks~$W$ and~$W'$ share only the vertices given by their first and last edge, that is we have~$V(W)\cap V(W')= s\cup t$. Observe that~$\abs{\cW}-pq\geq \frac{1}{12}\alpha n^{\ell-1}/\orderededge{e}(H)$. Note that for all subsets~$\cW'\subseteq \cW$ of size at most~$q$, there are at most~$\ell^2 q n^{\ell-k-2}\leq  n^{\ell-9/8}/\orderededge{e}(H)<\frac{1}{12}\alpha n^{\ell-1}/\orderededge{e}(H)$ walks~$W\in\cW$ with~$V(W)\cap V(W')\neq s\cup t$ for at least one walk~$W'\in\cW'$. These two observations imply that there are sets~$\cW_1,\ldots,\cW_{p}$ as described above (we may simply greedily construct them).
	
	A suitable union bound shows that it suffices to obtain 
	\begin{equation}\label{equation: bound for one set}
	\pr\Biggl[\abs{\{ W\in \cW':\text{$W'$ is~$\cT$-compatible}\} }\leq \frac{3}{4}\biggl(\frac{r}{n}\bigg)^{\ell-1}q\Biggr]\leq \exp(-n^{2/3})
	\end{equation}
	for all~$\cW'\in\{\cW_1,\ldots,\cW_p\}$.
	To this end, fix~$\cW'\in\{\cW_1,\ldots,\cW_p\}$. For~$W=\orderededge{e}_1\ldots\orderededge{e}_{\ell}\in \cW'$ and~$i\in[\ell-1]$, let~$X_{W,i}$ be the indicator random variable of the event that~$\orderededge{e}_{i}\orderededge{e}_{i+1}$ is~$\cT$-compatible. From the fact that the walks in~$\cW'$ are internally self-avoiding, the symmetry of the complete graphs whose vertices are the edges of~$H$ that contain~$\{s_2,\ldots,s_k\}$ and~$\{t_2,\ldots,t_k\}$, respectively,~$q\leq 2n^{3/4}$ and~$r\geq \eps n$, we obtain
	\begin{align*}
	\pr\Biggl[X_{W,1}=X_{W,\ell-1}=1\Biggm| (X_{W',i})_{\begin{subarray}{l}
		W'\in\cW'\setminus\{W\},\\i\in[\ell-1]
		\end{subarray}}\Biggr]&=\pr\Biggl[X_{W,1}=X_{W,\ell-1}=1\Biggm| (X_{W',i})_{\begin{subarray}{l}
		W'\in\cW'\setminus\{W\},\\i\in\{1,\ell-1\}
		\end{subarray}}\Biggr]\\
	&\geq \frac{5}{6}\biggl(\frac{r}{n}\biggr)^2
	\end{align*}
	for all~$W\in \cW'$.
	For~$W\in \cW'$, let~$X_W$ be the indicator random variable of the event that~$W$ is~$\cT$-compatible.
	Since the walks in~$\cW'$ are internally self-avoiding and since we have~$V(W_1)\cap V(W_2)\subseteq s\cup t$ for all~$W_1,W_2\in\cW'$, we conclude that
	\begin{equation*}
	\pr[X_W=1\mid (X_{W'})_{W'\in\cW'\setminus\{W\}}]\geq \frac{5}{6}\biggl(\frac{r}{n}\biggr)^{\ell-1}
	\end{equation*}
	holds for all~$W\in \cW'$. This shows that~$\sum_{W\in \cW'} X_W$ stochastically dominates a binomial random variable with parameters~$q$ and~$\frac{4}{5}\bigl(\frac{n}{r}\bigr)^{\ell-1}$ and implies~\eqref{equation: bound for one set} by using Chernoff's inequality (see Lemma~\ref{lemma: Chernoff}).
\end{proof}

\begin{lemma}\label{lemma: mixing}
	Let~$(X_t)_{t\in\bN_0}$ be a Markov chain with state space~$\{s_1,\ldots,s_n\}$, transition matrix~$(p_{i,j})_{i,j\in[n]}$ with~$p_{i,j}\neq 0$ for all~$i,j\in[n]$ and (unique) stationary distribution given by~$(\sigma_i)_{i\in[n]}$ with~$\sigma_i\neq 0$ for all~$i\in[n]$.
	Let~$\alpha:=\min_{i,j,k\in[n]}\frac{p_{i,j}}{\sigma_k}$ and~$\beta:=\max_{i,j,k\in[n]}\frac{p_{i,j}}{\sigma_k}$. Then
	\begin{equation*}
		\pr[X_t=s_i]=(1\pm(1-\alpha/2)^t)\sigma_i
	\end{equation*}
	holds for~$t\geq 2+2\alpha^{-1}\log\beta$.
\end{lemma}
\begin{proof}
	For~$t\geq 0$ and~$j\in[n]$, let~$\delta^j_{t}:=\pr[X_t=s_j]/\sigma_j$. We start the proof with finding a lower and an upper bound for~$\delta^j_{t+1}$ in terms of~$\alpha$,~$\beta$ and~$\delta^1_{t},\ldots,\delta^n_{t}$. Then we use these bounds and induction on~$t$ to prove the desired statement.
	
	For~$t\geq 0$, let~$\delta^{\min}_{t}:=\min_{j\in[n]}\delta^j_{t}$. 
	For all~$j\in[n]$, we may write
	\begin{equation*}
		\pr[X_{t+1}=s_j]
		=\sum_{i\in[n]} p_{i,j}(\delta^{\min}_{t}\sigma_i+\pr[X_t=s_i]-\delta^{\min}_{t}\sigma_i)
		=\delta^{\min}_{t}\sigma_j+\sum_{i\in[n]} p_{i,j}(\pr[X_t=s_i]-\delta^{\min}_{t}\sigma_i).
	\end{equation*}
	This yields
	\begin{equation}\label{equation: mixing lower bound induction step}
		\delta^j_{t+1}
		\geq \delta^{\min}_{t}+\sum_{i\in[n]} \alpha(\pr[X_t=s_i]-\delta^{\min}_{t}\sigma_i)	
		=\delta^{\min}_{t}+\alpha-\alpha\delta^{\min}_{t}
		=1-(1-\alpha)(1-\delta^{\min}_{t})
	\end{equation}
	and
	\begin{equation}\label{equation: mixing upper bound induction step}
		\delta^j_{t+1}\leq \delta^{\min}_{t}+\sum_{i\in[n]} \beta(\pr[X_t=s_i]-\delta^{\min}_{t}\sigma_i)	
		=\delta^{\min}_{t}+\beta-\beta\delta^{\min}_{t}=1+(\beta-1)(1-\delta^{\min}_{t}).
	\end{equation}
	We claim that
	\begin{equation}\label{equation: mixing lower bound}
		\delta^i_t\geq 1-(1-\alpha)^t.
	\end{equation}
	holds for all $t\geq 0$ and $i\in[n]$.
	Indeed, it holds for $t=0$.
	Suppose~\eqref{equation: mixing lower bound} holds for some $t_0\geq 0$ (and all $i\in [n]$).
	Then this implies $\delta^{\min}_{t_0}\geq 1-(1-\alpha)^{t_0}$
	and thus~\eqref{equation: mixing lower bound induction step} implies that~\eqref{equation: mixing lower bound} holds for~$t=t_0+1$ and all~$i\in[n]$.
	By induction, \eqref{equation: mixing lower bound} holds for all  $t\geq 0$ and $i\in[n]$.

	Now~\eqref{equation: mixing upper bound induction step},~\eqref{equation: mixing lower bound} and the observation that~$\beta\geq 1$\COMMENT{%
		Let~$\sigma_{\min}:=\min_{i\in[n]} \sigma_i$. $\beta<1$ implies~$p_{i,j}<\sigma_{\min}$ for all~$i,j\in[n]$ and thus~$n=\sum_{i\in[n]} 1=\sum_{i\in[n]}\sum_{j\in[n]} p_{i,j}<\sum_{i\in[n]}\sum_{j\in[n]} \sigma_{\min}=n\sum_{i\in[n]}\sigma_{\min}\leq n\sum_{i\in[n]}\sigma_{i}=n$.
	} imply
	\begin{equation*}
		\delta^{i}_t\leq 1+(\beta-1)(1-\delta^{\min}_{t-1})\leq 1+(\beta-1)(1-\alpha)^{t-1}
	\end{equation*}
	for all~$t\geq 1$ and~$i\in[n]$.
	If~$t\geq 2+2\alpha^{-1}\log\beta\geq 2+\log(1/\beta)/\log(1-\alpha/2)$\COMMENT{%
		$\eul x\leq \eul^x\Rightarrow \log(x)\leq x-1\Rightarrow \log(1-x)\leq -x\Rightarrow \log\bigl(\frac{1}{1-x}\bigr)\geq x$ and thus~$\frac{2\log(\beta)}{\alpha}=\frac{\log(\beta)}{\alpha/2}\geq \frac{\log(\beta)}{\log(1/(1-\alpha/2))}=\frac{\log(1/\beta)}{\log(1-\alpha/2)}$	
	}, then\COMMENT{%
		use $\beta(1-\alpha/2)^{t}\leq (1-\alpha/2)^{t-\log(1/\beta)/\log(1-\alpha/2)}\leq 1$
	}~$(\beta-1)(1-\alpha)^{t-1}\leq\beta(1-\alpha/2)^{2t-2}\leq (1-\alpha/2)^t$.
\end{proof}
For~$\zeta\in(0,1)$, we say that a~$k$-graph~$H$ on~$n$ vertices with~$r$-regular transition system~$\cT$ is~\defn{$\zeta$-exactly~$\cT$-compatibly~$\ell$-connected} if for all~$\orderededge{s},\orderededge{t}\in\orderededge{E}(H)$, the number of~$\cT$-compatible~$\ell$-walks from~$\orderededge{s}$ to~$\orderededge{t}$ in~$H$ is~$(1\pm\zeta) r^{\ell-1}/\orderededge{e}(H)$. 

\begin{lemma}\label{lemma: counting walks}
	Let~$1/n\ll \eps,1/k$. Suppose~$H$ is a~$k$-graph on~$n$ vertices and~$\cT$ is an~$r$-regular transition system of~$H$ with~$r\geq \eps n$ such that~$H$ is~$\cT$-compatibly~$(\alpha,\ell_0)$-connected.
	Then~$H$ is~$(1-\frac{\alpha}{3\ell_0})^{\ell}$-exactly~$\cT$-compatibly~$\ell$-connected for all~$\ell\geq3k\ell_0\log(2/\eps)/\alpha$.
\end{lemma}
\begin{proof}
	This is essentially a special case of Lemma~\ref{lemma: mixing}.
	Fix~$\ell$ as above and let~$\orderededge{s},\orderededge{t}\in\orderededge{E}(H)$ and~$w$ be the number of~$\cT$-compatible~$\ell$-walks from $\orderededge{s}$ to $\orderededge{t}$ in $H$.
	First, we argue why we essentially may assume that $\ell-1$ is a multiple of $\ell_0-1$.
	Let~$\ell'$ be the remainder in the division of~$\ell-1$ by~$\ell_0-1$. 
	The number of~$\cT$-compatible~$({\ell'}+1)$-walks starting at~$\orderededge{s}$ is~$r^{\ell'}$. 
	Thus it suffices to show that for every~$\orderededge{s}'$ that is the last element of such a walk, 
	or more generally every~$\orderededge{s}'\in\orderededge{E}(H)$,
	the number of~$(\ell-\ell')$-walks from~$\orderededge{s}'$ to~$\orderededge{t}$ is~$(1\pm (1-\frac{\alpha}{2\ell_0})^{\ell-\ell'}){r^{\ell-\ell'-1}}/{\orderededge{e}(H)}$ (note that $(1-\frac{\alpha}{2\ell_0})^{\ell-\ell'}\leq (1-\frac{\alpha}{3\ell_0})^\ell$\COMMENT{%
		$(1-\frac{\alpha}{3\ell_0})^{\ell-\ell'+\ell'}\geq (1-\frac{\alpha}{3\ell_0})^{3(\ell-\ell')/2}\geq (1-\frac{\alpha}{2\ell_0})^{\ell-\ell'}$
	}). 
	Therefore, from now on we assume~$(\ell_0-1)\mid (\ell-1)$ and show that~$w=(1\pm (1-\frac{\alpha}{2\ell_0})^{\ell})r^{\ell-1}/\orderededge{e}(H)$ holds.
	
	Consider a~$\cT$-compatible simple random walk~$(\orderededge{E}_t)_{t\in\bN_0}$ on~$\orderededge{E}(H)$ starting at~$\orderededge{s}$, that is we have~$\orderededge{E}_0=\orderededge{s}$ and~$\orderededge{E}_t$ with~$t\geq 1$ is iteratively chosen uniformly at random among all elements $\orderededge{e}\in\orderededge{E}(H)$ for which~$\orderededge{E}_0\ldots \orderededge{E}_{t-1}\orderededge{e}$ is a~$\cT$-compatible walk in~$H$.
	For all~$\orderededge{e},\orderededge{f}\in\orderededge{E}(H)$, the number of~$\cT$-compatible~$\ell_0$-walks from~$\orderededge{e}$ to~$\orderededge{f}$ is at least~$\alpha r^{\ell_0-1}/\orderededge{e}(H)$ and since the number of~$\cT$-compatible~$\ell_0$-walks starting at~$\orderededge{e}$ is~$r^{\ell_0-1}$, we conclude that for all~$t\geq 0$, we have
	\begin{equation*}
		\pr[\orderededge{E}_{t+\ell_0-1}=\orderededge{f}\mid \orderededge{E}_{t}=\orderededge{e}]\geq \frac{\alpha r^{\ell_0-1}}{\orderededge{e}(H)r^{\ell_0-1}}\geq \frac{\alpha}{\orderededge{e}(H)}
	\end{equation*}
	and similarly, since the number of~$\cT$-compatible~$\ell_0$-walks from~$\orderededge{e}$ to~$\orderededge{f}$ in~$H$ is at most~$r^{\ell_0-k-1}$, we obtain
	\begin{equation*}
		\pr[\orderededge{E}_{t+\ell_0-1}=\orderededge{f}\mid \orderededge{E}_{t}=\orderededge{e}]\leq \frac{r^{\ell_0-k-1}}{r^{\ell_0-1}}\leq \frac{1}{\eps^k n^k}\leq \frac{1}{\eps^k \orderededge{e}(H)}.
	\end{equation*}
	Since~$(\orderededge{E}_t)_{t\in\bN_0}$ is a Markov chain,~$(\orderededge{E}_{(\ell_0-1)t})_{t\in\bN_0}$ is also a Markov chain. Moreover, they have the same stationary distribution and since~$\cT$ is regular, this stationary distribution is given by~$(1/\orderededge{e}(H))_{i\in[\orderededge{e}(H)]}$.
	Note that the given lower bound for~$\ell$ implies\COMMENT{%
		use~$\frac{\ell-1}{\ell_0-1}\geq \frac{3k\ell_0\log(2/\eps)-1}{\ell_0\alpha}\geq \frac{3k\ell_0\log(2/\eps)-\ell_0/3}{\ell_0\alpha}\geq \frac{2k\ell_0\log(2/\eps)+(k-1/3)\ell_0/\log(2)}{\ell_0\alpha}\geq \frac{2k\ell_0\log(2/\eps)+(5/3)\ell_0/\log(2)}{\ell_0\alpha}\geq \frac{2\log(\eps^{-k})}{\alpha}+5\log(2)$
	}~$\frac{\ell-1}{\ell_0-1}\geq \frac{2\log(\eps^{-k})}{\alpha}+2$.
	Thus applying Lemma~\ref{lemma: mixing} to~$(\orderededge{E}_{(\ell_0-1)t})_{t\in\bN_0}$ with~$\alpha$ and~$\eps^{-k}$ playing the roles of~$\alpha$ and~$\beta$ yields
	\begin{equation*}
		w=\pr[\orderededge{E}_{\ell-1}=\orderededge{t}]r^{\ell-1}=\Bigl(1\pm(1-\alpha/2)^{\frac{\ell-1}{\ell_0-1}}\Bigr)\frac{r^{\ell-1}}{\orderededge{e}(H)}
	\end{equation*}
	which implies the desired bounds because\COMMENT{%
		use~$\frac{\ell-1}{\ell_0-1}=\frac{\ell\ell_0-\ell_0}{\ell_0^2-\ell_0}\geq \frac{\ell\ell_0-\ell}{\ell_0^2-\ell_0}=\frac{\ell(\ell_0-1)}{\ell_0(\ell_0-1)}=\frac{\ell}{\ell_0}$
	}
		$(1-\frac{\alpha}{2})^{\frac{\ell-1}{\ell_0-1}}
		\leq (1-\frac{\alpha}{2})^{\frac{\ell}{\ell_0}}
		\leq(1-\frac{\alpha}{2\ell_0})^{\ell}.$
\end{proof}

\section{Adjustments to the initial weight distribution}\label{sec:transporter}

Suppose $H$ is an~$(\alpha,\ell_0)$-connected~$k$-graph on~$n$ vertices. Lemmas~\ref{lemma: transition system connecting} and~\ref{lemma: counting walks} ensure the existence of a transition system $\cT$ of~$H$ such that every edge of $H$ is contained in approximately the same number of $\cT$-compatible~$\ell$-cycles in $H$ whenever~$\ell$ is large in terms of~$\ell_0$. 
Placing the same weight on every such cycle yields an initial weight distribution $\omega_0$ on the edges of~$H$ such that the weight $\sum_{C\in \cC_{\ell}(H)\colon e\in E(C)} \omega_0(e)$ on the edge~$e$ is roughly the same for all $e\in E(H)$, say after suitable normalisation, approximately~$1$. 
This section comprises the description of the refinement of $\omega_0$ by incremental modifications until we obtain a fractional~$C^k_{\ell}$-decomposition.
These modifications mimic sending small fractions of weight from one edge to another and are realized by manipulating the weights assigned to cycles that form certain substructures of $H$, which we call transporters.

For a~$k$-tuple~$\orderededge{e}=(e_1,\ldots,e_k)$ and an integer~$i\geq 0$, 
we define \defn{$\rotatededge{e}{i}:=(e_{i+1},\ldots,e_{i+k})$} with indices considered modulo $k$.
Let $\ell\geq 4$.
For $\orderededge{s},\orderededge{t}\in \orderededge{E}(H)$ with $s\cap t=\emptyset$, a closed $(k\ell+1) $-walk $\orderededge{e}_0\ldots\orderededge{e}_{k\ell}$ in $H$ is an \defn{$(\ell-1,\orderededge{s},\orderededge{t})$-transporter} in $H$ if
\begin{enumerate}[label=(\roman*)]
	\item\label{item: visits rotated} $\orderededge{e}_{i\ell}=\rotatededge{s}{i}$ and $\orderededge{e}_{i\ell+\floor{\ell/2}}=\rotatededge{t}{i}$ hold for all $i\in[k-1]_0$,
	\item the walk $\orderededge{e}_{i\ell}\ldots \orderededge{e}_{(i+1)\ell -k}$ is self-avoiding for all $i\in[k-1]_0$ and 
	\item $V(\orderededge{e}_{i\ell}\ldots \orderededge{e}_{(i+1)\ell -k})\cap V(\orderededge{e}_{j\ell}\ldots \orderededge{e}_{(j+1)\ell -k})\subseteq s\cup t$ holds for all distinct~$i,j\in[k-1]_0$.
\end{enumerate}
We say that a closed walk~$T$ in~$H$ is an~\defn{$(\ell-1)$-transporter} in~$H$ if there are~$\orderededge{s},\orderededge{t}\in\orderededge{E}(H)$ such that~$T$ is an $(\ell-1,\orderededge{s},\orderededge{t})$-transporter in~$H$. 
For these and the following definitions it is notationally more convenient to consider~$(\ell-1)$-transporters, however in the proof of Lemma~\ref{lemma: adjustments} we will use~$\ell$-transporters instead.

Let~$T=\orderededge{e}_0 \ldots \orderededge{e}_{k\ell}$ be an $(\ell-1,(s_1,\ldots,s_k),(t_1,\ldots,t_k))$-transporter in $H$. Observe that property~\ref{item: visits rotated} implies that for all~$i\in[k-1]_0$, there are~$u,v,u',v'\in V(H)$ such that 
\begin{equation*}
\begin{aligned}
\orderededge{e}_{i\ell+\floor{\ell/2}-1}&=(u,t_{i+1},\ldots,t_{i+k-1}),&&&\orderededge{e}_{(i-1)\ell+\floor{\ell/2}+1}&=(t_{i+1},\ldots,t_{i+k-1},v),\\
\orderededge{e}_{(i+1)\ell-1}&=(u',s_{i+2},\ldots,s_{i+k})&\text{ and }&&\orderededge{e}_{i\ell+1}&=(s_{i+2},\ldots,s_{i+k},v')
\end{aligned}
\end{equation*}
with indices of~$\orderededge{e}$ considered modulo~$k\ell$ and indices of~$s$ and~$t$ considered modulo~$k$.
This allows the following definition of two crucial sets of cycles associated with~$T$ as illustrated in Figure~\ref{figure:trasporter}:
The \defn{sending cycles} of $T$ are the~$(\ell-1)$-cycles given by the self-avoiding closed walks
\begin{equation*}
\orderededge{e}_{i\ell} \ldots \orderededge{e}_{i\ell+\floor{\ell/2}-1}\orderededge{e}_{(i-1)\ell+\floor{\ell/2}+1} \ldots \orderededge{e}_{i\ell}
\end{equation*}
and the \defn{receiving cycles} of $T$ are the~$(\ell-1)$-cycles given by the self-avoiding closed walks
\begin{equation*}
\orderededge{e}_{i\ell+\floor{\ell/2}} \ldots \orderededge{e}_{(i+1)\ell-1}\orderededge{e}_{i\ell+1} \ldots \orderededge{e}_{i\ell+\floor{\ell/2}}
\end{equation*}
with indices considered modulo $k\ell$ and $i\in[k-1]_0$. Observe that every edge~$e\in E(T)\setminus\{s,t\}$ is an edge of exactly one sending and one receiving cycle while~$s$ is an edge of exactly~$k$ sending and zero receiving cycles and~$t$ an edge of exactly zero sending and~$k$ receiving cycles of~$T$.
\begin{figure}[t]
	\centering
	\begin{tikzpicture}[scale=1, x={(0.95cm, 0cm)}, y={(0cm,0.75cm)}]\footnotesize

	\pgfdeclarelayer{marker}
	\pgfsetlayers{marker,main}
	
	\pgfkeys{%
		/tikz/layer/.code={
			\pgfonlayer{#1}\begingroup
			\aftergroup\endpgfonlayer
			\aftergroup\endgroup
		}
	}

	\tikzset{%
		vertex/.style={fill=black, circle, inner sep=0pt, outer sep=0pt, minimum size=4pt},
		many/.style={decorate, decoration={zigzag, pre length=5pt, post length=5pt, segment length=8pt, amplitude=2pt}},
		arc/.style={-{Stealth[width=5pt, angle=50:5pt]}},
		dashed/.style={dash pattern=on 2.5pt off 2.5pt},
		blue marker/.style={line width=10.5pt, line cap=round, line join=round, color=blue, opacity=0.3},
		red marker/.style={line width=6.5pt, line cap=round, line join=round, color=red, opacity=0.3}
	}
	
	\foreach [count=\c] \i in {1,...,4} {
		\pgfmathsetmacro\x{3*(\c-1)}
		\node[vertex] (u\i) at (\x,3) {};
		\node[vertex] (v\i) at (\x,-3) {};
	}
	\foreach [count=\c] \i in {j,k} {
		\pgfmathsetmacro\x{8+3*(\c-1)}
		\node[vertex] (u\i) at (\x,3) {};
		\node[vertex] (v\i) at (\x,-3) {};
	}
	
	\path (u1.center)--(v1.center) node(u1+1)[pos=0.25, vertex]{} node(v1-1)[pos=0.875, vertex]{};
	\foreach \i in {2,...,4,j} {
		\path (u\i.center)--(v\i.center) node(u\i+1)[pos=0.25, vertex]{} node(v\i-1)[pos=0.75, vertex]{};
	}
	\path (uk.center)--(vk.center) node(uk+1)[pos=0.125, vertex]{} node(vk-1)[pos=0.75, vertex]{};
	
	\foreach \i/\j in {1/2,2/3,3/4,j/k} {
		\path (v\i.center)--(u\j.center) node(v\i+1)[pos=0.25, vertex]{} node(u\j-1)[pos=0.75, vertex]{};
	}	
	\path (vk.center)--(u1.center) node(vk+1)[pos=0.125, vertex]{} node(u1-1)[pos=0.875, vertex]{};
	
	\foreach \i in {1,2,3,k} {
		\draw[arc, many] (u\i+1)--(v\i-1);
		\draw[arc] (u\i)--(u\i+1);	
		\draw[arc] (v\i-1)--(v\i);
	}
	
	\foreach \i/\j in {1/2,2/3,3/4,j/k} {
		\draw[arc, many] (v\i+1)--(u\j-1);
	}
	
	
	\foreach \i in {1,2,3,k} {
		\draw[arc] (v\i)--(v\i+1);
		\draw[arc] (u\i-1)--(u\i);
	}
	
	\foreach \i/\j in {1/2,2/3,3/4,j/k} {
		\draw[arc, dashed] (u\j-1)--(u\i+1);
		\draw[arc, dashed] (v\j-1)--(v\i+1);
	}
	
	\draw[draw=none, fill=white] (7.75, 3.9) rectangle (9.25, -3.9);
	
	\foreach \i/\label in {
		1/$\orderededge{s}=\orderededge{e}_0$,
		2/$\rotatededge{s}{1}=\orderededge{e}_{\ell}$,
		3/$\rotatededge{s}{2}=\orderededge{e}_{2\ell}$,
		k/$\rotatededge{s}{k-1}=\orderededge{e}_{(k-1)\ell}$} {
		
		\node[above] at (u\i.north) {\label};
	}

	\foreach \i/\label in {
		1/$\orderededge{t}=\orderededge{e}_{\floor*{\ell/2}}$,
		2/$\rotatededge{t}{1}=\orderededge{e}_{\ell+\floor*{\ell/2}}$,
		3/$\rotatededge{t}{2}=\orderededge{e}_{2\ell+\floor*{\ell/2}}$,
		k/$\rotatededge{t}{k-1}=\orderededge{e}_{(k-1)\ell+\floor*{\ell/2}}$} {
		
		\node[below, yshift=-2pt] at (v\i.south) {\label};
	}

	\foreach \i/\j in {1/2,2/3,3/4,j/k,k/1} {
		\draw[red marker, layer=marker] (u\j.center)--(v\j-1.center)--(v\i+1.center)--cycle;
	}

	\foreach \i/\j in {1/2,2/3,3/4,j/k,k/1} {
		\draw[blue marker, layer=marker] (v\i.center)--(u\j-1.center)--(u\i+1.center)--cycle;
	}
	
	\begin{scope}
		\clip (7.75, 3.9) rectangle (9.25, -3.9);
		
		\draw[red marker, layer=main] (u1.center)--(v1-1.center)--(vk+1.center)--cycle;
		
		\draw[blue marker, layer=main] (vk.center)--(u1-1.center)--(uk+1.center)--cycle;
	\end{scope}
	
	\draw[arc, many] (vk+1)--(u1-1);
	\draw[arc, dashed] (u1-1)--(uk+1);
	\draw[arc, dashed] (v1-1)--(vk+1);
	
	\draw[-] (7.75, 3.9) -- (7.75, -3.9);
	\draw[-] (9.25, 3.9) -- (9.25, -3.9);
	\node[rotate=0] at (8.5, 0) {\makebox[1em][c]{.\hfil.\hfil.}};
	
	\end{tikzpicture}
	\caption{For $\ell\geq 4$, a $k$-graph $H$ and $\protect\orderededge{s},\protect\orderededge{t}\in\protect\orderededge{E}(H)$ with $s\cap t=\emptyset$:
	 $(\ell-1,\protect\orderededge{s},\protect\orderededge{t})$-transporter $T=\protect\orderededge{e}_0\ldots\protect\orderededge{e}_{k\ell}$ in $H$ visualized as a directed cycle in the directed graph $\protect D$ with vertex set $\protect\orderededge{E}(H)$ whose edges are the pairs $(\protect\orderededge{e},\protect\orderededge{f})\in\protect\orderededge{E}(H)^2$ for which $\protect\orderededge{e}\protect\orderededge{f}$ is a walk in $H$.
	Vertical zigzag arrows represent directed $(\floor{\ell/2}-2)$-paths in $D$ and diagonal zigzag arrows 
	represent directed $(\ceil{\ell/2}-2)$-paths in $D$.
	Dashed arrows indicate the shortcuts taken for the sending cycles (in red) and receiving cycles (in blue).
	}
	\label{figure:trasporter}
\end{figure}

Let $\omega\colon \cC_{\ell-1}(H)\rightarrow\bR$ and $w\geq 0$. We say the \defn{function obtained from $\omega$ by using $T$ with weight $w$} is the function $\omega'\colon \cC_{\ell-1}(H)\rightarrow\bR$ with
\begin{equation*}
\omega'(C)= \begin{cases}
\omega(C)-\frac{w}{k}&\text{if $ C $ is a sending cycle of $ T $;}\\
\omega(C)+\frac{w}{k}&\text{if $ C $ is a receiving cycle of $ T $;}\\
\omega(C)&\text{otherwise}
\end{cases}
\end{equation*}
for all~$C\in\cC_{\ell-1}(H)$. 
Note that as a consequence of the observation above, using~$T$ with weight~$w$ shifts weight~$w$ from~$s$ to~$t$ while the weight on all other edges of~$H$ remains unchanged.

In the following we use transporters to shift weight between edges. In order to obtain in the end a fractional decomposition, we have to ensure that all cycle weights remain non-negative. Thus our use of a transporter for weight transportation is constrained by the weights on its sending cycles.

The goal of the refinement of~$\omega_0$ is a redistribution of the 
deviations of the weights on the edges from (their target value)~$1$.
Consider these deviations as vertex weights given by~$\xi\colon V\rightarrow\bR$ in an auxiliary directed graph~$D$ with vertex set~$V$ in which possible weight transportations through the use of transporters are represented by arcs. 
We adjust our weights as follows and as detailed in the proof of Lemma~\ref{lemma: transport plan}.
For all distinct~$x,y\in V$, send weight~${\xi(x)}/{\abs{V}}$ from~$x$ to~$y$ in~$D$ spread equally among directed paths from~$x$ to~$y$ that have length~$\ell$ (we will apply the statement with $\ell=2$).
\begin{lemma}\label{lemma: transport plan}
	Suppose~$\ell\geq 1$ is an integer, $\alpha>0$ and~$D$ is a directed graph on~$n$ vertices with vertex set~$V$ and arc set~$A$ with~$(v,u)\in A$ for all~$(u,v)\in A$. Suppose that for all distinct~$s,t\in V$, there are at least~$\alpha n^{\ell-1}$ directed~$\ell$-paths from~$s$ to~$t$ in~$D$. Let~$\beta> 0$ and~$\xi\colon V\rightarrow [-\beta,\beta]$ with~$\sum_{v\in V} \xi(v)=0$. Then there is a function~$\eta\colon A\rightarrow \bigl[0,\frac{2\beta\ell}{\alpha n}\bigr]$ such that
	\begin{equation*}
		\xi(v)+\sum_{u:(u,v)\in A} \eta(u,v)-\sum_{u:(v,u)\in A} \eta(v,u)=0
	\end{equation*}
	holds for all~$v\in V$.
\end{lemma}
\begin{proof}
	We may think of $\xi(v)$ as a (potentially negative) weight located at $v$ and we want to find a (non-negative) flow $\eta$ that distributes these weights such that afterwards the weight located at each vertex is~$0$.
	To this end, for all distinct vertices~$u$ and~$v$, let $\cP_{u,v}$ be a set of $ \ceil{\alpha n^{\ell-1}}$ (directed) $\ell$-paths from~$u$ to~$v$ in $D$.
	As a first step, we define a flow $\eta'\colon A\rightarrow \bR$ which may also send negative weight through arcs.
	Our strategy is as follows; every vertex $v$ sends weight $\frac{\xi(v)}{n}$ to each vertex~$u\neq v$ by sending weight~$\frac{\xi(v)}{n\abs{\cP_{v,u}}}$ along every path in~$\cP_{v,u}$.
	Thus every vertex $v$ sends weight $\bigl(1- \frac{1}{n}\bigr)\xi(v)$ to other vertices.
	How much weight does $v$ receive?
	From each vertex~$u\neq v$, it receives $\xi(u)/n$ and so in total $\sum_{u\in V\setminus \{v\}}\xi(u)/n=-\xi(v)/n$.
	Therefore, this achieves the desired weight distribution.
	
	As each arc of $D$ is contained in at most $\ell n^{\ell-1}$ paths of length $\ell$ in $D$ and as we send weight of absolute value at most $\frac{\beta}{\alpha n^{\ell}}$ along each path, we observe that $|\eta'(x,y)|\leq \frac{\beta \ell}{\alpha n}$ holds for all~$(x,y)\in A$.
	
	Now, we simply transform a negative flow~$\eta'(x,y)$ on an arc~$(x,y)$ into a positive flow in the opposite direction by setting $\eta(x,y)=\max\{\eta'(x,y)-\eta'(y,x),0\}$ for all~$(x,y)\in A$.
	As $|\eta(x,y)|\leq |\eta'(x,y)|+|\eta'(y,x)|\leq \frac{2\beta \ell}{\alpha n}$, this completes the proof.
\end{proof}

\begin{lemma}\label{lemma: adjustments}
	Let~$1/n\ll \eps,\zeta,1/k,1/\ell$. 
	Suppose~$H$ is a~$k$-graph on~$n$ vertices with~$\delta(H)\geq \delta n$ and~$\cT$ is an~$r$-regular transition system of~$H$ with~$r\geq \eps n$ such that~$H$ is~$(1-\zeta)^{\ell'}$-almost~$\cT$-compatibly~$\ell'$-connected for all integers~$\ell'$ with~$\ell/3\leq \ell'\leq 2\ell$. 
	Let~$\mu\in(2^{-\ell},1)$ and suppose~$\ell(1-\zeta/2)^{\ell+1}\leq \frac{\delta^{k+1}\mu}{400k}$ holds. 
	Then there is a fractional~$C_{\ell}^k$-decomposition~$\omega$ of~$H$ with~$\omega(C)=(1\pm \mu) \frac{2e(H)}{r^{\ell}}$ for all~$\cT$-compatible~$C\in\cC_{\ell}(H)$ and~$\omega(C)\leq \mu\frac{2e(H)}{r^{\ell}}$ for all~$C\in\cC_{\ell}(H)$ that are not~$\cT$-compatible.
\end{lemma}
\begin{proof}
	We will start with a scaled uniform distribution~$\omega_0$ on all~$\cT$-compatible~$\ell$-cycles in~$H$ and show that~$\omega_0$ is already almost a function~$\omega$ as desired (see~\eqref{equation: initial value is close} and~\eqref{equation: initial weight sum is close} below). 
	In a further step we will slightly adjust~$\omega_0$ using transporters to obtain a fractional~$C^k_{\ell}$-decomposition of~$H$.
	
	Let us observe that for all integers~$\ell'\geq 1$ with~$\ell/3\leq \ell'\leq 2\ell$,~$U\subseteq V(H)$ with~$\abs{U}\leq k\ell$ and~$\orderededge{s},\orderededge{t}\in \orderededge{E}(H-U)$, the number of internally self-avoiding~$\cT$-compatible~$\ell'$-walks from~$\orderededge{s}$ to~$\orderededge{t}$ in~$H-U$ is
	\begin{equation}\label{equation: estimation of number of walks}
		(1\pm(1-\zeta/2)^{\ell'})\frac{r^{\ell'-1}}{\orderededge{e}(H)}.
	\end{equation}
	Indeed, the number of~$\ell'$-walks from~$\orderededge{s}$ to~$\orderededge{t}$ in~$H$ that are not internally self-avoiding is at most
	\begin{equation*}
		4\ell^2 r^{\ell'-k-2}\leq \frac{4\ell^2}{\eps^{k+1}n}\frac{r^{\ell'-1}}{n^k}\leq\frac{1}{\sqrt{n}}\frac{r^{\ell'-1}}{\orderededge{e}(H)}
	\end{equation*}
	and the number of~$\ell'$-walks from~$\orderededge{s}$ to~$\orderededge{t}$ in~$H$ that are not walks in~$H-U$ is at most
	\begin{equation*}
		2k\ell^2r^{\ell'-k-2}\leq\frac{2k\ell^2 }{\eps^{k+1}n}\frac{r^{\ell'-1}}{n^k}\leq\frac{1}{\sqrt{n}}\frac{r^{\ell'-1}}{\orderededge{e}(H)}.
	\end{equation*}
	
	Note that for every fractional~$C_{\ell}^k$-decomposition~$\omega$ of~$H$, we have~$\sum_{C\in\cC_{\ell}(H)}\omega(C)=e(H)/\ell$. Let~$c_{\cT,\ell}$ be the number of~$\cT$-compatible~$\ell$-cycles in~$H$ and let~$\omega_0\colon \cC_{\ell}(H)\rightarrow [0,\infty)$ such that
	\begin{equation*}
		\omega_0(C)=\begin{cases}
			\frac{e(H)}{\ell c_{\cT,\ell}}&\text{if~$C$ is~$\cT$-compatible};\\
			0&\text{otherwise}
		\end{cases}
	\end{equation*}
	holds for all~$C\in \cC_{\ell}(H)$.
	When counting the cycles given by the self-avoiding closed~$\cT$-compatible~$(\ell+1)$-walks from~$\orderededge{e}$ to~$\orderededge{e}$ in~$H$ for all~$\orderededge{e}\in\orderededge{E}(H)$, we count every~$\cT$-compatible~$\ell$-cycle in~$H$ exactly~$2\ell$ times. Together with~\eqref{equation: estimation of number of walks} this shows
	\begin{equation}\label{equation: estimate number of cycles}
	c_{\cT,\ell}=(1\pm(1-\zeta/2)^{\ell+1})\frac{r^{\ell}}{\orderededge{e}(H)}\cdot \orderededge{e}(H)\cdot \frac{1}{2\ell}=(1\pm(1-\zeta/2)^{\ell+1})\frac{r^{\ell}}{2\ell}.
	\end{equation}
	As~$(1-\zeta/2)^{\ell+1}\leq \mu/4$, this implies\COMMENT{%
		use~$\frac{1}{1-x}= \frac{1-x^2}{1-x}+\frac{x^2}{1-x}\leq 1+x+2x^2\leq 1+2x$ and~$\frac{1}{1+x}\geq \frac{1-x^2}{1+x}=1-x$ with~$x=(1-\zeta/2)^{\ell+1}$
	}
	\begin{equation}\label{equation: initial value is close}
		\frac{e(H)}{\ell c_{\cT,\ell}}=\left(1\pm \frac{\mu}{2}\right)\frac{2e(H)}{r^{\ell}}.
	\end{equation}
	For~$e\in E(H)$, let~$\xi(e)$ be the weight by which the sum of the weights of the~$\cT$-compatible $\ell$-cycles~$C$ in~$H$ with~$e\in E(C)$ deviates from $1$, that is, let~$\xi\colon E(H)\rightarrow \bR$ such that 
	\begin{equation*}
		\xi(e)=\Bigl(\sum_{C\in \cC_{\ell}(H)\colon e\in E(C)}\omega_0(C)\Bigr)-1
	\end{equation*}
	holds for all~$e\in E(H)$.
	When counting the cycles given by the self-avoiding closed~$\cT$-compatible~$(\ell+1)$-walks from~$\orderededge{f}$ to~$\orderededge{f}$ in~$H$ for all~$\orderededge{f}\in \orderededge{E}(H)$ with~$f=e$, we count every~$\cT$-compatible~$\ell$-cycle~$C$ in~$H$ with~$e\in E(C)$ exactly twice.
	Consequently, since~$(1-\zeta/2)^{\ell+1}\leq 1/2$, exploiting~\eqref{equation: estimation of number of walks} and~\eqref{equation: estimate number of cycles} yields\COMMENT{%
		use~$\frac{1+x}{1-x}\leq 1+\frac{2x}{1-x}\leq 1+4x$ and~$\frac{1-x}{1+x}\geq 1-\frac{2x}{1+x}\geq 1-4x$ with~$x=(1-\zeta/2)^{\ell+1}$
	}
	\begin{equation}\label{equation: initial weight sum is close}
		\xi(e)=(1\pm(1-\zeta/2)^{\ell+1})\frac{r^{\ell}}{\orderededge{e}(H)}\cdot k!\cdot\frac{1}{2}\cdot\frac{e(H)}{\ell c_{\cT,\ell}}-1=0\pm 4(1-\zeta/2)^{\ell+1}.
	\end{equation}
	To obtain from Lemma~\ref{lemma: transport plan} the existence of a flow describing adjustments of~$\omega_0$ that transform it into a function~$\omega$ as desired, consider the directed graph~$D$ with vertex set~$E(H)$ where a pair~$(e,f)$ of edges of~$H$ is an arc in~$D$ if and only if~$e\cap f=\emptyset$. Let~$\xi_{\max}:=\max_{e\in E(H)}\abs{\xi(e)}$. Observe that~$e(H)\geq \frac{\delta n^k}{2k!}$.
	For all distinct~$s,t\in E(H)$, there are at least~$e(H)/2$ edges in~$H$ disjoint from~$s$ and~$t$ and thus there are at least~$e(H)/2$ (directed) paths from~$s$ to~$t$ in~$D$ that have length~$2$. Consequently, Lemma~\ref{lemma: transport plan} yields a function~$\eta\colon A(D)\rightarrow\bigl[0,\frac{16k!\xi_{\max}}{\delta n^k}\bigr]$ with
	\begin{equation}\label{equation: adjustments work}
		\sum_{C\in \cC_{\ell}(H)\colon e\in E(C)}\omega_0(C)+\sum_{f\colon(f,e)\in A} \eta(f,e)-\sum_{f\colon(e,f)\in A} \eta(e,f)=1
	\end{equation}
	for all~$e\in E(H)$.
	(Note that we now consider~$\ell$-transporters, not~$(\ell-1)$-transporters as in the discussion above.)
	Observe that we can build an~$(\ell,\orderededge{s},\orderededge{t})$-transporter~$\orderededge{e}_0\ldots\orderededge{e}_{k(\ell+1)}$ whose sending cycles are~$\cT$-compatible by choosing~$\orderededge{e}_{i(\ell+1)+\floor*{\frac{\ell+1}{2}}+1}$ and concatenating~$\cT$-compatible self-avoiding~$\ceil*{\frac{\ell+1}{2}}$-walks from~$\orderededge{e}_{i(\ell+1)+\floor*{\frac{\ell+1}{2}}+1}$ to~$\rotatededge{s}{i+1}$ and $\cT$-compatible self-avoiding~$\bigl(\floor*{\frac{\ell+1}{2}}+1\bigr)$-walks from~$\rotatededge{s}{i+1}$ to~$\orderededge{e}_{i(\ell+1)+\floor*{\frac{\ell+1}{2}}+1}$ (with indices considered modulo~$k(\ell+1)$) for all~$i\in[k-1]_0$. 
	As we know how many such walks exist at least, this implies that for all~$\orderededge{s},\orderededge{t}\in \orderededge{E}(H)$ with~$s\cap t=\emptyset$, the number of~$(\ell,\orderededge{s},\orderededge{t})$-transporters in~$H$ whose sending cycles are~$\cT$-compatible is at least\COMMENT{%
		since~$\ell\geq k$ holds we have~$(1-\zeta/2)^{\ell+1}\leq \frac{1}{100k^2}$ and thus~$(1-\zeta/2)^{\frac{\ell+1}{2}}\leq \frac{1}{10k}$. This yields~$(1-(1-\zeta/2)^{\ceil*{\frac{\ell+1}{2}}})^{k}(1-(1-\zeta/2)^{\floor*{\frac{\ell+1}{2}}+1})^{k}\geq (1-\frac{1}{6k})^{2k}\geq 1-\frac{2k}{6k}=2/3$
	}
	\begin{align*}
		m&:=\left((\delta n-k\ell)\cdot\left(1-(1-\zeta/2)^{\ceil*{\frac{\ell+1}{2}}}\right)\frac{r^{\ceil*{\frac{\ell+1}{2}}-1}}{\orderededge{e}(H)}\cdot\left(1-(1-\zeta/2)^{\floor*{\frac{\ell+1}{2}}+1}\right)\frac{r^{\floor*{\frac{\ell+1}{2}}}}{\orderededge{e}(H)}\right)^k\\
		&\geq\frac{1}{2}\delta^k n^k\frac{r^{k\ell}}{\orderededge{e}(H)^{2k}}.
	\end{align*}
	For~$\orderededge{s},\orderededge{t}\in \orderededge{E}(H)$ with~$s\cap t=\emptyset$, let~$\sT_{\orderededge{s},\orderededge{t}}$ be a set of~$m$ distinct~$(\ell,\orderededge{s},\orderededge{t})$-transporters in~$H$ whose sending cycles are~$\cT$-compatible. Let~$T_1\ldots T_p$ be an ordering of
	\begin{equation*}
		\bigcup_{\orderededge{s},\orderededge{t}\in \orderededge{E}(H)\colon s\cap t=\emptyset}\sT_{\orderededge{s},\orderededge{t}}
	\end{equation*}
	and for~$i\in[p]$, let~$\omega_i$ be the function obtained from~$\omega_{i-1}$ by using the~$(\ell,\orderededge{s}_i,\orderededge{t}_i)$-transporter~$T_i$ with weight~$\frac{\eta(s,t)}{(k!)^2m}$. 
	(Note that for all disjoint~$s,t\in E(H)$, the number of pairs~$(\orderededge{s}',\orderededge{t}')$ with~$\orderededge{s}',\orderededge{t}'\in\orderededge{e}(H)$,~$s'=s$ and~$t'=t$ is~$(k!)^2$.) 
	Let~$\omega:=\omega_p$. By~\eqref{equation: adjustments work}, the function~$\omega$ satisfies
	\begin{equation*}
		\sum_{C\in \cC_{\ell}(H)\colon e\in E(C)}\omega(C)=1
	\end{equation*}
	for all~$e\in E(H)$ and it remains to show~$\omega(C)\geq 0$ for all~$C\in\cC_{\ell}(H)$ and $\omega(C)=(1\pm \mu) \frac{2e(H)}{r^{\ell}}$ for all~$C\in\cC_{\ell}(H)$ that are~$\cT$-compatible as well as~$\omega(C)\leq\mu \frac{2e(H)}{r^{\ell}}$ for all~$C\in\cC_{\ell}(H)$ that are not~$\cT$-compatible.
	
	For all~$C\in\cC_{\ell}(H)$, there are at most~$2k\ell\cdot n$ pairs~$(\orderededge{s},\orderededge{t})\in\orderededge{E}(H)^2$ such that~$C$ is a sending cycle of an~$(\ell,\orderededge{s},\orderededge{t})$-transporter in~$H$ whose sending cycles are~$\cT$-compatible (given~$C$, we have~$2k\ell$ choices for~$\orderededge{s}$ and then at most~$n$ choices for~$\orderededge{t}$). Now we argue similarly as above and conclude that for all~$C\in\cC_{\ell}(H)$, the number of~$\ell$-transporters in~$H$ whose sending cycles are~$\cT$-compatible that have~$C$ as one of their sending cycles is bounded from above by\COMMENT{%
		since~$\ell\geq k$ holds we have~$(1-\zeta/2)^{\ell+1}\leq \frac{1}{100k^2}$ and thus~$(1-\zeta/2)^{\frac{\ell+1}{2}}\leq \frac{1}{10k}$. This yields~$(1+(1-\zeta/2)^{\ceil*{\frac{\ell+1}{2}}})^{k}(1+(1-\zeta/2)^{\floor*{\frac{\ell+1}{2}}+1})^{k}\leq (1+\frac{1}{6k})^{2k}=((1+\frac{1}{6k})^{6k})^{1/3}\leq ((\eul^{1/6k})^{6k})^{1/3}\leq 3/2$
	}
	\begin{align*}
		m'&:=2k\ell n\left(n\cdot\left(1+(1-\zeta/2)^{\ceil*{\frac{\ell+1}{2}}}\right)\frac{r^{\ceil*{\frac{\ell+1}{2}}-1}}{\orderededge{e}(H)}\cdot\left(1+(1-\zeta/2)^{\floor*{\frac{\ell+1}{2}}+1}\right)\frac{r^{\floor*{\frac{\ell+1}{2}}}}{\orderededge{e}(H)}\right)^{k-1}\\&\leq 2k\ell n^{k}\frac{3}{2} \frac{r^{k\ell-\ell}}{\orderededge{e}(H)^{2k-2}}.
	\end{align*}
	Furthermore, similar considerations show that~$m'$ is also an upper bound for the number of~$\ell$-transporters in~$H$ whose sending cycles are~$\cT$-compatible that have~$C$ as one of their receiving cycles.
	
	In the transition from~$\omega_0$ to~$\omega$, by employing transporters as suggested by~$\eta$, we used transporters with weight at most~$\frac{1}{(k!)^2 m}\cdot\frac{16k!\xi_{\max}}{\delta n^k}=\frac{16\xi_{\max}}{\delta k! mn^k}$. Recall that using a transporter with weight~$w$ changes the weights on its sending and receiving cycles by~$w/k$. Thus exploiting~\eqref{equation: initial value is close} yields that it suffices to show that
	\begin{equation*}
		m'\cdot\frac{1}{k}\cdot\frac{16\xi_{\max}}{\delta k! m n^k}\leq \frac{\mu}{2}\frac{2e(H)}{r^{\ell}}.
	\end{equation*}	
	We obtain
	\begin{equation*}
		m'\cdot\frac{1}{k}\cdot\frac{16\xi_{\max}}{\delta k! m n^k}
		\leq \frac{6\ell\orderededge{e}(H)^2}{\delta^k r^{\ell}}\cdot \frac{16\xi_{\max}}{\delta k! n^k}
		\stackrel{\eqref{equation: initial weight sum is close}}{\leq} \frac{192}{\delta^{k+1}}\cdot \ell(1-\zeta/2)^{\ell+1}\cdot \frac{2e(H)}{r^{\ell}}\leq \frac{\mu}{2}\frac{2e(H)}{r^{\ell}},
	\end{equation*}
	which completes the proof.
\end{proof}

\section{Obtaining an almost uniform factional decomposition}\label{sec:proof main thm}

Observe that we immediately can deduce from Lemmas~\ref{lemma: transition system connecting},~\ref{lemma: counting walks} and \ref{lemma: adjustments} 
that all sufficiently large $(\alpha,\ell_0)$-connected~$k$-graphs admit a fractional $\ell$-cycle decomposition
whenever~$\ell$ is large in terms of~$\ell_0$.
However, in the following proof of Theorem~\ref{thm: stronger version}, we additionally exploit that Lemma~\ref{lemma: transition system connecting} does not simply guarantee the existence of a transition system suitable for the application of Lemma~\ref{lemma: counting walks} and thus Lemma~\ref{lemma: adjustments}, but that a randomly chosen transition system is suitable with high probability.

\lateproof{Theorem~\ref{thm: stronger version}}
First we show that Lemma~\ref{lemma: transition system connecting} and a probabilistic argument yield multiple suitable transition systems as a basis for an application of Lemmas~\ref{lemma: counting walks} and \ref{lemma: adjustments}. Then we use the fractional decompositions obtained from Lemma~\ref{lemma: adjustments} to construct a fractional decomposition as desired.

Suppose~$1/n\ll \alpha,\mu,1/\ell,1/k$. Suppose~$H$ is an~$(\alpha,\ell_0)$-connected~$k$-graph as in Theorem~\ref{thm: stronger version}.
Let~$\delta:=\delta(H)/n$,~$\Delta:=\Delta(H)/n$,~$\eps:=(1-\mu/4)^{1/\ell}\delta$,~$r\geq 2$ be an even integer with~$\eps n\leq r\leq \delta n$ and~$\zeta:=\frac{\alpha}{6\ell_0}$. 
Let~$\cT_1,\ldots,\cT_n$ be independently chosen random~$r$-transition systems of~$H$.
For~$\orderededge{s}=(s_1,\ldots,s_k)\in \orderededge{E}(H)$ with~$d_H(\{s_2,\ldots,s_k\})=\delta(H)$ and~$\orderededge{t}\in\orderededge{E}(H)$, the number of~$\ell_0$-walks from~$\orderededge{s}$ to~$\orderededge{t}$ in~$H$ is at most~$\delta(H)n^{\ell_0-k-2}$.
Thus~$\delta\geq \alpha$ holds.
The lower bound for~$\ell$ implies
\begin{equation*}
	\ell/3\geq \frac{60k\ell_0}{\alpha}\log\biggl(\frac{3}{\alpha}\biggr)\geq \frac{60k\ell_0}{\alpha}\log\biggl(\frac{3}{\delta}\biggr)= \frac{30k\ell_0}{\alpha/2}\log\biggl(3(1-\mu/4)^{1/\ell}\frac{1}{\eps}\biggr)\geq \frac{3k\ell_0\log(2/\eps)}{\alpha/2}.
\end{equation*}
Thus Lemmas~\ref{lemma: transition system connecting} and~\ref{lemma: counting walks} imply
\begin{equation*}
	\pr[\text{$H$ is~$(1-\zeta)^{\ell'}$-exactly~$\cT_i$-compatibly~$\ell'$-connected for all~$\ell'\geq\ell/3$}]\geq 1-\exp(-\sqrt{n})
\end{equation*}
for all~$i\in[n]$.
For all~$C\in\cC_{\ell}(H)$ and~$i\in[n]$, the cycle~$C$ is~$\cT_i$-compatible with probability at least~$\bigl(\frac{\eps n}{\Delta(H)}\bigr)^{\ell}=(1-\mu/4)\delta^{\ell}/\Delta^{\ell}$.
Consequently, Lemma~\ref{lemma: Chernoff} and a suitable union bound show that it is possible to select~$\cT_1,\ldots,\cT_n$ with
\begin{equation*}
	\abs{\{ i\in[n]:\text{ $C$ is~$\cT_i$-compatible } \}}\geq (1-\mu/2)\frac{\delta^{\ell}}{\Delta^{\ell}}n
\end{equation*}
for all~$C\in\cC_{\ell}(H)$ such that~$H$ is~$(1-\zeta)^{\ell'}$-exactly~$\cT_i$-compatibly~$\ell'$-connected for all~$\ell'\geq\ell/3$ and~$i\in[n]$. 
The function~$x\mapsto x\exp(-x)$ is monotonically decreasing on~$[1,\infty)$\COMMENT{%
	the function given by~$x\exp(-x)$ has its maximum at~$x=1$ and is thus monotonically decreasing for~$x\geq 1$
}. 
In addition, with room to spare, we will exploit that  $\exp(-15kAB)\leq \exp(-11k)\exp(-3kA)\exp(-kB)$ for $A,B\geq 1$, where we set $A:=\log\frac{\ell_0}{\alpha}$ and $B:=\log\frac{1}{\mu}$.
This implies
\begin{align*}
	\ell(1-\zeta/2)^{\ell+1}
	&\leq \frac{2}{\zeta}\cdot\frac{\zeta}{2}\ell\exp\biggl(-\frac{\zeta}{2}\ell\biggr)
	\leq \frac{2}{\zeta}\cdot 15k\log\biggl(\frac{\ell_0}{\alpha}\biggr)\log\biggl(\frac{1}{\mu}\biggr)\exp\biggl(-15k\log\biggl(\frac{\ell_0}{\alpha}\biggr)\log\biggl(\frac{1}{\mu}\biggr)\biggr)\\
	&\leq 180k\frac{\ell_0}{\alpha}\log\biggl(\frac{\ell_0}{\alpha}\biggr)\log\biggl(\frac{1}{\mu}\biggr) 
	\cdot \exp(-11k)\biggl(\frac{\alpha}{\ell_0}\biggr)^{3k}\mu^{k}\\
	&\leq \frac{1}{800k}\biggl(\frac{\alpha}{\ell_0}\biggr)^{k+1}\mu
	\leq \frac{\delta^{k+1}\mu/2}{400k}.
\end{align*}
Thus, Lemma~\ref{lemma: adjustments} yields fractional~$C^k_{\ell}$-decompositions~$\omega_1,\ldots,\omega_n$ of~$H$ with
\begin{equation*}
	\abs*{\Bigl\{ i\in[n]:{\textstyle \omega_i(C)\geq (1-\mu/2)\frac{2e(H)}{r^{\ell}}} \Bigr\}}\geq (1-\mu/2)\frac{\delta^{\ell}}{\Delta^{\ell}}n\quad\text{ and }\quad\omega_i(C)\leq (1+\mu/2)\frac{2e(H)}{r^{\ell}}
\end{equation*}
for all~$C\in\cC_{\ell}(H)$ and~$i\in[n]$.
For~$\omega=\frac{1}{n}\sum_{i\in[n]} \omega_i$ and~$C\in\cC_{\ell}(H)$ the first of these two inequalities implies
\begin{equation*}
	\omega(C)\geq (1-\mu/2)^2\frac{\delta^{\ell}}{\Delta^{\ell}}\frac{2e(H)}{r^{\ell}}\geq (1-\mu/2)^2\frac{2e(H)}{\Delta(H)^{\ell}}\geq (1-\mu)\frac{2e(H)}{\Delta(H)^{\ell}}
\end{equation*}
and from the second we obtain\COMMENT{%
	For the last inequality use~$\frac{1+x}{1-x/2}\leq 1+x\frac{3/2}{1-x/2}\leq 1+x\frac{3/2}{3/4}\leq 1+2x$ for~$x=\mu/2$
}
\begin{equation*}
	\omega(C)\leq (1+\mu/2)\frac{2e(H)}{r^{\ell}}\leq (1+\mu/2)\frac{2e(H)}{\eps^{\ell}n^{\ell}}= \frac{1+\mu/2}{1-\mu/4}\cdot\frac{2e(H)}{\delta(H)^{\ell}}\leq (1+\mu)\frac{2e(H)}{\delta(H)^{\ell}}.
\end{equation*}
\endproof

\section{Concluding remarks}

In this paper we prove the existence of fractional cycle decompositions in $k$-graphs~$H$ on~$n$ vertices under the very mild assumption that $H$ is ``well-connected''.
This in particular includes the case that $\delta(H)\geq (1/2+o(1))n$,
settles a question posed by Glock, K\"uhn and Osthus for graphs,
and improves significantly on previous results.

The method presented here can be easily pushed further; for example for graphs, with few modifications one can show that 
the fractional decomposition threshold for
subdivisions of cliques (or even any connected graph)
also tends to $1/2$ when the length of the subdivided edges tends to infinity.

Another natural question arising from Theorem~\ref{thm: main result} is the best dependence between $\eps$ and $\ell$ (for fixed $k$) where $\eps$ is defined by $\delta_{C_\ell^k}^*=1/2+\eps$.
Example~\ref{example} shows that $\eps\geq \frac{1}{(k-1+2/k)(\ell-1)}$
and our results imply, say, that $\eps = O(\ell^{-1/(4k!)})$.
Any substantial improvement on these bounds would be interesting.

\bibliographystyle{amsplain}
\bibliography{ReferencesLocal}

\bigskip

\end{document}